\documentclass[11pt]{article}\textwidth 160mm\textheight 235mm
\usepackage[table,xcdraw]{xcolor}
\usepackage{amsfonts}
\usepackage{amssymb}
\usepackage{graphicx}
\usepackage{amsmath}
\usepackage{amsthm}
\usepackage{enumitem}
\usepackage{tikz}
\usepackage{float}
\usepackage{cancel}
\usepackage{todonotes}
\usepackage{tcolorbox}
\tcbuselibrary{skins}
\usepackage{lipsum}
\usetikzlibrary{matrix}
\usetikzlibrary{arrows,shapes}
\usepackage{cite}
\usepackage{hyperref}
\hypersetup{colorlinks=true, urlcolor=blue}
\expandafter\let\expandafter\oldproof\csname\string\proof\endcsname
\let\oldendproof\endproof
\renewenvironment{proof}[1][\proofname]{%
  \oldproof[\ttfamily \scshape \bf #1. ]%
}{\oldendproof}
\def\O{{\cal O}}
\def\S{{\mathbb{S}}}

\def\B{\mathbb{B}}
\def\R{{\rm I\!R}}
\def\N{{\rm I\!N}}
\def\ox{\bar{x}}
\def\oy{\bar{y}}

\def\ow{\bar{w}}

\def\ss{\scriptsize }

\def\X{{\mathbb X}}

\def\hat{\widehat}

\def\dist{{\rm dist}}

\def\tto{\rightrightarrows}

\def\d{{\rm d}}

\def\sub{\partial}

\def\disp{\displaystyle}

\def\Bar{\overline}
\def\ra{\rangle}
\def\la{\langle}

\def\co{\mbox{\rm co}\,}
\def\cone{\mbox{\rm cone}\,}

\def\epi{\mbox{\rm epi}\,}
\def\dim{\mbox{\rm dim}\,}
\def\dom{\mbox{\rm dom}\,}

\def\ker{\mbox{\rm ker}\,}

\def\proj{\mbox{\rm proj}\,}

\def\argmin{\mbox{\rm argmin}\,}

\def\Lim{\mbox{\rm Lim}\,}
\def\Limsup{\mbox{\rm Lim sup}\,}

\def\dn{\downarrow}
\def\O{\Omega}

\def\st{\stackrel}
\def\oR{\Bar{\R}}

\def\dd{\delta}

\def\sce{\setcounter{equation}{0}}

\begin{document}
\vspace*{0.5in}
\begin{center}
{\bf PENALITY METHODS TO COMPUTE STATIONARY SOLUTIONS IN CONSTRAINED OPTIMIZATION PROBLEMS }\\[1 ex]
ASHKAN MOHAMMADI \footnote{Department of Mathematics and Statistics, Georgetown University, Washington, DC 20057, USA (ashkan.mohammadi@georgetown.edu).}

\end{center}
\vspace*{0.05in}
\small{\bf Abstract.} This paper is devoted to studying the stationary solutions of a general constrained optimization problem through its associated unconstrained penalized problems. We aim to answer the question, \'' what do the stationary solutions of a penalized unconstrained problem tell us about the stationary solutions of the original constrained optimization problem?". We answer the latter question by establishing relationships between global (local) minimizers and stationary points of the two optimization problems. Given the strong connection between stationary solutions between problems, we introduce a new approximate $\varepsilon$-stationary solution for the constrained optimization problems. We propose an algorithm to compute such an approximate stationary solution for a general constrained optimization problem, even in the absence of Clarke regularity. Under reasonable assumptions, we establish the rate $O(\varepsilon^{-2})$ for our algorithm, akin to the gradient descent method for smooth minimization. Since our penalty terms are constructed by the powers of the distance function, our stationarity analysis heavily depends on the generalized differentiation of the distance function. In particular, we characterize the (semi-)differentiability of the distance function $\mbox{dist} (. ;X)$ defined by a Fr\'echet smooth norm, in terms of the geometry of the set $X$. We show that $\mbox{dist} (. ;X)$ is semi-differentiable if and only if $X$ is geometrically derivable. The latter opens the door to design optimization algorithms for constrained optimization problems that suffer from Clarke irregularity in their objectives and constraint functions. 
\\[1ex]


\newtheorem{Theorem}{Theorem}[section]
\newtheorem{Proposition}[Theorem]{Proposition}
\newtheorem{Remark}[Theorem]{Remark}
\newtheorem{Lemma}[Theorem]{Lemma}
\newtheorem{Corollary}[Theorem]{Corollary}
\newtheorem{Definition}[Theorem]{Definition}
\newtheorem{Example}[Theorem]{Example}
\newtheorem{Algorithm}[Theorem]{Algorithm}
\renewcommand{\theequation}{{\thesection}.\arabic{equation}}
\renewcommand{\thefootnote}{\fnsymbol{footnote}}

\normalsize

\section{Introduction}\sce
Penalty methods are among the most important methods in solving constrained optimization problems. Namely, a constrained optimization problem is replaced by a sequence of sub-problems in which the constraints are represented by terms added to the objective function. Usually, these subproblems are unconstrained optimization problems or constrained optimization problems with an easy constraint set, in the sense that projection onto the constraint set admits a closed-form. There are different ways to construct the penalty term; a widely used one is the smooth (quadratic) penalty. More precisely, in the framework of non-linear programming 
\begin{equation}\label{nlp}
\mbox{minimize} \;\;  \varphi (x)  \quad \mbox{subject to}\;\; g_i (x)  \leq 0 ~ i=1,...,l,  \quad h_j (x) = 0 ~ j=l+1,...,m
\end{equation}
the quadratic penalty associated with \eqref{nlp} is given by the following unconstrained problem
\begin{equation}\label{qpnlp}
	\mbox{minimize} \;\;  \varphi (x) + \rho \sum_{i=1}^{l} \max ^{~ ~~~2} \{\ g_i (x), 0\}  +  \rho \sum_{j=l+1}^{m}  h^2_j (x) \quad \mbox{subject to}\;\; x \in \R^n
\end{equation}
where $\rho > 0$ is the penalty parameter, whose job is to intensify the price of the constraint violation. Now the question is how close we can get to the solution of \eqref{nlp} by solving \eqref{qpnlp}.  It is well-known that a global minimizer of \eqref{nlp} can be achieved as a limiting point of the global minimizers of \eqref{qpnlp} when $\rho \to + \infty$. More precisely, if $x_k$ is a global minimizer of the unconstrained problem \eqref{qpnlp} with the penalty parameter $\rho_k$, then every limit point of the sequence $\{x_k\}_{k \in  \N}$ is a global minimizer of \eqref{nlp} provided that $\lim_{\substack{k\to\infty}} \rho_k = \infty$ ; see \cite[Theorem~17.1]{noc}. In an abstract framework, the above result is followed by the epi-convergence argument due to Rockafellar and Wets \cite{rw}. Namely, the sequence of smooth functions $$f_k (x) :=  \varphi (x) + \rho_k \sum_{i=1}^{l} \max ^{~ ~~~2} \{\ g_i (x), 0\}  +  \rho_k \sum_{j=l+1}^{m}  h^2_j (x)$$ epi-converges to the extended-real-valued function $$f(x):= \varphi (x) +  \sum_{i=1}^{l}   \delta_{-} (g_i (x))  +  \sum_{j=l+1}^{m}  \delta_{0} (h_j (x)),$$ where $\delta_{-} , \;  \delta_{0} : \R \to (-\infty , 0]$ are the indicator functions of the sets $(-\infty , \infty]$ and $\{0\}$ respectively, hence, any limit point of the sequence $x_k \in \argmin f_k (x)$ is a global minimizer of \eqref{nlp} by \cite[Theorem~7.33]{rw}. The latter results can be slightly improved by allowing $x_k$ be a sub-optimal provided that a constraint qualification is assumed; see \cite[Theorem~17.2]{noc} and \cite[Theorem~7.33]{rw} . Although the problem \eqref{qpnlp} is an unconstrained optimization problem, finding an (approximate) global optimal solution for \eqref{qpnlp} is impractical for non-convex problems. However, there are many fast optimization algorithms to find an (approximate) stationary point for the non-convex problem \eqref{qpnlp}; see \cite{ noc, fp, cp}. Motivated by the practicality of the obtaining an approximate stationary point for the unconstrained problem \eqref{qpnlp}, in this paper, we aim to establish a relationship between the approximate stationary points of problems \eqref{qpnlp} and \eqref{nlp}. More precisely, in Theorem \ref{ss}, we show that, if $x_k$ is a $\varepsilon$-stationary point of \eqref{qpnlp}, i.e., $\| \nabla f_{k} (x_k) \| \leq \epsilon$, then for sufficiently large $k$, $x_k$ is an approximate stationary point for \eqref{nlp}, in the sense that there exist so-called \textit{approximate} Lagrangian multipliers $\lambda_k^i$ and $\mu_k^j$ such that  
\begin{equation*}\label{appkkt}
\left\{\begin{matrix}
	\| \nabla \varphi(x_k) + \sum_{i=1}^{l} \lambda^i_k  \nabla g_i (x_k) + \sum_{j=l+1}^{m}   \mu^j_k \nabla h_j (x_k)\| \leq \varepsilon & \quad \mbox{ approximate stationarity} \\\\
	g_i(x_k) \leq \varepsilon ~i=1,...,l,\quad |  h_j (x_k) |  \leq \varepsilon ~ j=l+1,...,m& ~\mbox{approximate feasibility} \quad \\\\
	\lambda_k^i \geq  - \varepsilon ~i=1,...,l,\quad   \mu^j_k  \in \R ~ j=l+1,...,m& \quad ~~~~\mbox{approximate dual feasibility}  \\
\end{matrix}\right.
\end{equation*}
Note that finding a stationary point for \eqref{qpnlp} is equivalent to solving the equation $\nabla f_k (x) = 0$, which can be done by several ways; e.g., applying the gradient descent method on the unconstrained problem \eqref{qpnlp}, or directly applying the semi-smooth Newton method on the equation $\nabla f_k (x) = 0$; see \cite{fp}. Notice that since the smooth function $x \to \max^2\{x,0\}$ is not twice differentiable at origin,  the mapping $x \to \nabla f_k (x)$ is not necessarily differentiable. Therefore, for solving $\nabla f_k (x) =0$, one cannot apply the classical smooth Newton method. It is natural to ask if one can truly obtain a stationary point of \eqref{nlp} by solving only one equation $\nabla f_k (x)$? In another word, if (under some conditions) any of $x_k$ can itself be a stationary point of \eqref{nlp}, i.e., $x_k$ be a feasible solution for \eqref{nlp} satisfying KKT conditions. It turns out that, by a quick computation, if any $x_k$ is a feasible solution for \eqref{nlp}, then $x_k $ can\textbf{not} be a stationary solution for \eqref{nlp} unless $\nabla \varphi (x_k) = 0.$ It is quite often that, in a constrained optimization problem, no feasible solution vanishes the gradient of the objective function. Consequently, all stationary solutions of \eqref{qpnlp}, for any $\rho$, remain infeasible for \eqref{nlp}. The latter is the main disadvantage of the quadratic penalty, and it won't be corrected unless one entirely changes the penalty term. Therefore, it is common for one uses the exact non-smooth penalty. Namely
\begin{equation}\label{epnlp}
	\mbox{minimize} \;\;  f_{\rho} (x) := \varphi (x) + \rho \sum_{i=1}^{l} \max  \{\ g_i (x), 0\}  +  \rho \sum_{j=l+1}^{m}  | h_j (x) | \quad \mbox{subject to}\;\; x \in \R^n
\end{equation}
It is well-known that if  $\ox$ is a local minimizer for \eqref{nlp} then, under a constraint qualification, $\ox$ is a local minimizer of \eqref{epnlp} for sufficiently large $\rho$. The latter implication helps derive necessary optimality conditions for \eqref{nlp}, and it has been well-investigated in the literature for different classes of optimization problems; see \cite{jzz, mm21, i, noc, hko}. However, for a numerical purpose, it is important to investigate the relationship between stationary solutions of problems \eqref{nlp} and \eqref{epnlp}. Indeed, by a quick computation, one can observe that if such a $x_{\rho}$ happens to be an (approximate) stationary solution for \eqref{epnlp} and feasible for the problem \eqref{nlp}, then $x_{\rho}$ is indeed an (approximate) stationary point for \eqref{nlp}; see Theorem \ref{spp} and Example \ref{egkkt}. Motivated by this observation,  in this paper, we investigate the sufficient conditions forcing $x_{\rho}$ to fall into the feasible solution set when $\rho >0$ is chosen sufficiently large. The rest of the paper is organized as follows. In section \ref{sect02}, we review the variational analysis tools needed for our analysis, and we investigate the generalized differential properties of distance functions; in particular, we calculate the subderivative of the distance functions defined by Fr\'echet smooth norms. In section \ref{sec03}, we consider exact non-smooth penalties. We investigate the link between stationary solutions of constrained optimization problems and their associated unconstrained penalized problems. In section \ref{sec04}, we propose a new definition of approximate stationary solutions not requiring feasibility. Later in section \ref{sec04}, we propose a first-order optimization algorithm to compute such approximate stationary solutions.  
   
\section{Preliminaries for first-order variational analysis }\sce \label{sect02}
Throughout this paper, we mostly use the standard notations of variational analysis and generalized differentiation; see, e.g. \cite{m18,rw}. For a nonempty set $\Omega\subset X$, the notation $x\st{\O}{\to}\ox$ indicates that $x\to\ox$ with $x\in\O$, while $\co\O$ and $\cone\O$ stand for the convex and conic hulls of $\O$, respectively. The indicator function $\dd_\O$ of a set $\O$ is defined by $\dd_\O(x):=0$ for $x\in\O$ and $\dd_\O(x):=\infty$ otherwise, while the Euclidean distance between $x\in\X$ and $\O$ is denoted by $\dist(x;\O)$. Sometimes we need to use $\ell_1$-norm for the distance of two sets, in latter case, we again use the same term ''$\dist $", but we clarify the use of $\ell_1$ norm.  For a closed set $\O$, we  denote $\proj_{\O} (x)$, the Euclidean projection on the set $\O$. We write $x=o(t)$ with $x\in\X$ and $t\in\R_+$ indicating that $\frac{\|x\|}{t}\dn 0$ as $t\dn 0$, where $\R_+$ (resp.\ $\R_-$) means the collection of nonnegative (resp.\ nonpositive) real numbers. Recall also that $\N:=\{1,2,\ldots\}$.
Given a nonempty set $\O \subset \R^n $ with $\ox\in \O $, the  tangent cone $T_{\O}(\ox)$ to $\O$ at $\ox$ is defined by
\begin{equation*}\label{2.5}
T_{\O}(\ox)=\big\{w\in \R^n |\;\exists\,t_k{\dn}0,\;\;w_k\to w\;\;\mbox{ as }\;k\to\infty\;\;\mbox{with}\;\;\ox+t_kw_k\in  \O \big\}.
\end{equation*}
\begin{Definition}[\bf geometrically derivable sets\cite{rw}] 
A tangent vector $w \in T_{X}(\ox)$ is said to be \textit{derivable}, if there exist $\varepsilon >0 $ and $\zeta : [0 , \varepsilon] \to X$ with $\zeta(0) = \ox$ and $\zeta'_{+} (0) = w$. The set $X$ is called geometrically derivable at $\ox$ if every tangent vector to $X$ at $\ox$ is derivable. We say $X$ is geometrically derivable if it is geometrically derivable at any $x \in X$.   
\end{Definition}
The class of geometrically derivable sets is considerably large. It is not difficult to check that the derivability is preserved under finite unions. Therefore, the finite union of convex sets is geometrically derivable as every convex set is geometrically derivable. A consequence of the latter result is that a non-convex polyhedron is geometrically derivable. Furthermore, under mild constraint qualification, the derivability is preserved under intersections and pre-image of smooth mappings \cite[Theorem~4.3]{mm21}. The derivability of $X$ at $\ox$ can equivalently be described by the relation $T_{X}(x) = \Lim_{t \dn 0} \frac{X - \ox}{t}$, where the limit means the set-limit in the sense of set-convergence \cite[p.~152]{rw}. Recall from \cite{mms1} that $f \colon\R^n\to\oR$ is {\em Lipschitz continuous} around $\ox\in\dom f$ {\em relative} to some set $\O\subset\dom f$ if there exist a constant $\ell\in\R_+$ and a neighborhood $U$ of $\ox$  such that
\begin{equation*}
| f(x)- f(u)|\le\ell\|x-u\|\quad\mbox{for all }\quad x,u\in\O\cap U.
\end{equation*}
 
The main advantage of defining Lipschitz continuity \textbf{relative} to a set, is not missing the extended-real-valued functions for the study of constrained optimization. 
Piecewise linear-quadratic functions and indicator functions of nonempty sets are important examples of extended-real-valued functions that are Lipschitz continuous relative to their domains around any point $\ox\in\dom f$. Similarly we can define the calmness of $f$ \textbf{at} point $\ox$ relative to a set. In particular, $f$ is called calm at $\ox$ from below if there exist the $\ell > 0$ and a neighborhood $U$ of $\ox$ such that
\begin{equation*}
 f(x)- f(\ox) \geq  - \ell\|x- \ox\|\quad\mbox{for all }\quad x \in U.
\end{equation*}
If $f$ is Lipschitz countinuous around $\ox$ relative to its domain, then $f$ is calm at $\ox$ from below. Given a function  $f:\R^n \to \oR$ and a point $\ox$ with $f(\ox)$ finite, the subderivative function $\d f(\ox)\colon\R^n\to[-\infty,\infty]$ is defined by
\begin{equation}\label{subderivatives}
	{\mathrm d}f(\ox)(w):=\liminf_{\substack{
			t\dn 0 \\
			w' \to w
	}} {\frac{f(\ox+tw')-f(\ox)}{t}}.
\end{equation}
Subderivative and tangent vectors are related by the relation $\epi \d f(\ox) = T_{\ss \epi f} (\ox , f(\ox))$. It is clear, from the latter tangent cone relationship, that $\d f(\ox)(.)$ is always lower semicontinuous, also $\d f(\ox)$ is proper if $f$ is calm at $\ox$ from below. Furthermore, We have $\dom \d f(\ox) = T_{\ss \dom f} (\ox )$ if $f$ is Lipschitz continuous around $\ox$ relative to $\dom f$, \cite[Lemma~4.2]{mm21}. If $f$ is convex or concave and $\ox \in \mbox{int} \; \dom f$ then \eqref{subderivatives} reduces to the usual directional derivative i.e.,
\begin{equation}\label{dderivative}
{\mathrm d}f(\ox)(w)= f' (x ; w) : =\lim_{\substack{
   t\dn 0 \\
    }} {\frac{f(\ox+tw)-f(\ox)}{t}}.
\end{equation}
Recall from \cite[Definition~7.20]{rw} that a vector-valued function $F:  \R^n \to \R^m$  is semi-differrentiable at $\ox$ for (direction) $w$, if the following limit exists 
\begin{equation}\label{semiderivatives}
{\mathrm d}F(\ox)(w):=\lim_{\substack{
   t\dn 0 \\
  w' \to w
  }} {\frac{F(\ox+tw')-F(\ox)}{t}}.
\end{equation}
$F$ is called semi-differentiable at $\ox$, if the above limit exists for all $w \in \R^n$, in this case, the continuous mapping $\d F(\ox) : \R^n \to \R^m$ is called the semi-derivative of $F$ at $\ox$. We say $F$ is semi-differentiable (without mentioning $\ox$) if $F$ is semi-differentiable everywhere in $\R^n$. As it is often clear from context, let us not use different notations for subderivative and semi-derivative, we use $\d$ for both them. Following \cite[Theorem~7.21]{rw} $F$ is semi-differentiable at $\ox$ if and only if there exists $h: \R^n \to \R^m$, continuous and homogeneous of degree one such that 
\begin{equation}\label{semidiff-ex}
F(x) = F(\ox) + h(x - \ox) + o (\| x - \ox \|),
\end{equation}
iindeed  $h = \d F(\ox)$. Observe from \eqref{semidiff-ex} that the difference between Fr\'echet differentiability and semi-differentiability is the linearity of the $h$. More precisely, $F$ is Fr\'echet differentiable at $\ox$ if and only if $F$ is semi-differentiable at $\ox$ and $\d F(\ox)(.)$ is a linear function. It is easy to see that the semi-differentiability and classical directional differentiability concepts are equivalent for the locally Lipschitz continuous functions. In contrast with \cite[p.~156]{cp}, in the definition of semi-differentiability, we do not require $F$ to be locally Lipschitz continuous. However, one can observe that if $F$ is semi-differentiable at $\ox$, then $F$ is calm at $\ox$, that is there exists a $\ell > 0$ such that for all $x$ sufficiently close to $\ox$
$$   \| F(x)  -  F(\ox)\|  \leq \ell \; \| x -  \ox \|  .$$
In the following remark, we highlight the advantage of the subderivative over the classical directional derivative.  
\begin{Remark}[\bf subderivative vs. classical directional derivative]\label{subrel}{\rm
It is straightforward to check that if $f$ is Lipschitz continuous around $\ox $ relative to $\O$ with a constant $\ell > 0$, then $ | \d f(x)(w)| \leq \ell \; \|w\| $ for all $w \in T_{\O} (\ox)$. This result will no longer remain true if one uses the classical directional derivative instead of subderivative; e.g.,  take $ f(x , y) = 0 $ for  all $(x,y) \in \O:= \{(x,y) \big|\;  x^2 \leq y \}$ and $f(x,y) = x$ otherwise. In the this example, $f$ is directionally differentiable everywhere and Lipschitz continuous relative to $\O$ with a constant $\ell = \frac{1}{2}$, but $f' ((0,0) ; (1, 0)) = 1$. In \cite{mf}, we established a rich calculus for subderivative making it a distinguished generalized derivative tool in variational analysis. Unfortunately, the classical directional derivative, as a less subtle derivative, do not provide such calculus in the absence of Lipschitz continuity. For instance, chain rule fails for the classical directional derivative of the composition $f \circ F$ at $x=0$ in direction $w=1$ where $f$ is chosen the aforementioned function and the smooth function $F(x):=(x,x^2)$. }
\end{Remark}
Given a positive number $\varepsilon >0 $, a function  $f:\R^n \to \oR$, and a point $\ox$ with $f(\ox)$ finite, the Fr\'echet subdifferential and $\varepsilon$-Fr\'echet subdifferential of $f$ is defined respectively by 
$$  \hat{\sub} f(\ox) := \{ v \in \R^n \; \big|~\la v , w\ra \leq \d f(\ox)(w), \quad \forall w \in \R^n  \} $$
$$  \hat{\sub}_{\varepsilon}  f(\ox) := \{ v \in \R^n \; \big|~\la v , w\ra - \varepsilon \| w\| \leq \d f(\ox)(w), \quad \forall w \in \R^n  \} $$
Recall from \cite{mm21} that $f$ is called Dini-Hadamard regular at $\ox$ if $\d f(\ox)(w) = \sup_{v \in \hat{\sub} f(\ox)} \la v ,w \ra.$ Following \cite{mm21}, Dini-Hadamard regularity is implied by Clarke regularity if $f$ is assumed locally Lipschitz continuous. In the following lemma we point out to an important consequence of Dini-Hadamard regularity. 
\begin{Lemma}\label{epsilonenlarge} Let $f : \R^n \to \R$ be Lipschitz countinuous around $\ox$ and  Dini-Hadamard regular at $\ox$. Then, $  \hat{\sub}_{\varepsilon}  f(\ox)  =  \hat{\sub}f(\ox) + \varepsilon   \B$. 
\end{Lemma}
\begin{proof}
First, without using any assumptions, we show that 	$ \hat{\sub}_{\varepsilon}  f(\ox) + \varepsilon   \B \subseteq \hat{\sub}_{\varepsilon}  f(\ox) .$ To prove the latter pick $v \in \hat{\sub}_{\varepsilon}  f(\ox)$  and $b \in \B$. For all $w \in \R^n$, we have  
$$  \la v+ \varepsilon b , w \ra  - \varepsilon \|w\| \leq  \la v ,  w \ra+  \varepsilon (1-\|b\|)\|w\|  \leq  \la v ,  w \ra \leq \d f(x) (w), $$
which yields $ v+ \varepsilon b  \in \hat{\sub}_{\varepsilon}  f(\ox) .$ To prove the other inclusion, pick $v \in \hat{\sub}_{\varepsilon}  f(\ox)$ then $  \la v , w \ra  - \varepsilon \|w\| \leq  \la \d f(x) (w)$ equivalently $\la v , w \ra \in \d f(\ox) (w) + \varepsilon \d \| . - \ox \| (\ox)(w) $, thus by subderivative sum rule in \cite[Theorem~4.4]{mm21}, we have  $\la v , w \ra \in \d \big( f + \varepsilon \| . - \ox \| \big) (\ox)(w) $ meaning that $v \in  \hat{\sub} \big( f + \varepsilon \| . - \ox \| \big) (\ox).$ Since  both $f$  and $\| . - \ox \|$ are Dini-Hadamard regular at $\ox$, by the subdifferential sum rule in  \cite[Corollary~5.3]{mm21}, we have 
$$v \in  \hat{\sub} f(\ox) + \varepsilon  \hat{\sub} \big( \| . - \ox \| \big) (\ox) =  \hat{\sub} f(\ox) + \varepsilon B.$$
\end{proof}

For given set $\O$ and a point$\ox \in \O$, the Fr\'echet normal cone at $\ox$ to the s$\O$ is defined by $\hat{N}_{\O} (\ox) := \hat{\sub} \delta_{\O} (\ox)$ that is
$$  \hat{N}_{\O} (\ox) = T^{*}_{\O}(\ox) :=\{ v \in \R^n \; \big|~\la v , w\ra \leq 0,  \quad \forall w \in T_{\O}(\ox)  \} $$
Next, we are going to review the subderivative and subdifferential chain rules for the composite function $f: = g \circ F $. We need to consider a qualification condition to achieve this, as we deal with extended-real-valued functions.
\begin{Definition}[\bf metric subregularity qualification condition] Let the composite function $f := g \circ F$ be finite at $\ox$, where $g : \R^m \to \oR$ and $F: \R^n \to \R^m  $. It is said that the metric subregularity qualification condition (MSQC) holds at $\ox$ for the composite function $f = g \circ F$ if there exist a $\kappa >0$ and $U$ an open neighborhood of $\ox$ such that for all $x \in U$ the following error bound holds :
\begin{equation}\label{msqc}
\dist ( x \; ; \; \dom f ) \leq \kappa \; \dist (F(x) \;  ; \; \dom g  ) 
\end{equation}
Similarly, we say the metric subregularity constraint qualification condition holds for the constraint set $\O  :=  \{ x \in \R^n | \;  F(x) \in X  \}$ at $\ox$ if MSQC holds for the composite function $\delta_{X} \circ F$, that is there exist a $\kappa > 0$ and $U$, an open neighborhood of $\ox$, such that for all $x \in U$ the following estimates holds:
\begin{equation}\label{msqccons}
\dist ( x \; ; \; \O ) \leq \kappa \; \dist (F(x) \;  ; \; X  )
\end{equation}
\end{Definition}
The above version of metric subregularity condition \eqref{msqc} for the composite function $g \circ F$ was first introduced in \cite[Definition~3.2]{mms1} by the author, which only concerns the domain of the functions $f$ and $g$, not their epigraphs as it was considered in \cite{io}. It is clear that the metric subregularity qualification condition is a robust property, in the sense that if \eqref{msqc} holds at $\ox$ then it also holds at any point sufficiently close to $\ox$. Also, observe that the metric subregularity qualification is invariant under changing space's norm. The latter is because all norms are equivalent in finite-dimensional vector spaces. The metric subregularity condition \eqref{msqc} has shown itself a reasonable condition under which the first and second-order calculus rules hold for amenable settings both in finite and infinite-dimensional spaces; see \cite{mm21, mms1, ms, mms2}. Metric subregularity constraint qualification is implied by Robinson constraint qualification. In particular, in the nonlinear programming, it is implied by Mangasarian Fromovitz constraint qualification, and in convex programming, it is implied by Slater constraint qualification; for the comparison of various types of constraint qualifications; see \cite[Proposition~3.1]{mms1} and \cite[Proposition~3.1]{mm21}. 
\subsection{Calculus rules via subderivative}

In the following theorem, we recall the subderivative and subdifferential which were recently established in \cite{mf} and \cite{mms1}. 
\begin{Theorem}[\bf subderivative chain rule]\label{fcalc}
Let $f: = g \circ F $ be finite at $\ox$ where $g: \R^m \to \oR$ is Lipschitz continuous relative to $\dom g$ and $F:\R^n \to \R^m $ is semi-differentiable at $\ox $ . Further, assume the metric subregularity qualification condition \eqref{msqc} holds. Then, the following subderivative chain rule holds
\begin{equation}\label{fcalc1}
\d f(\ox) (w) = \d g ( F(\ox)) ( \d F(\ox) w)  \quad \quad \forall w \in \R^n  .
\end{equation}
Additionally, if $g$ is convex and $F$ is continuously differetiable around $\ox$, then  $g \circ F$ is Dini-Hadamard regular at $\ox $, and the following subdifferential chain rule holds:
\begin{equation}\label{schian}
	\hat{\sub} f(\ox) =    \nabla F(\ox)^{*}  \;  \hat{\sub} g ( F(\ox)) 
\end{equation}
\end{Theorem}
\begin{proof} The subderivatiive chain rule \eqref{fcalc1} recently was established in \cite[Theorem~3.4]{mf}. For the proof of the subdifferential chain rule see \cite[Theorem~3.6]{mms1}.
\
\end{proof}
In Theorem \ref{fcalc}, if $g$ is semi-differentiable  at $F(\ox)$, then subderivative chain rule \eqref{fcalc1} holds without assuming Lipschitz continuity. The proof in the latter case is straightforward, just like the chain rule for differentiable functions. Therefore, we present the following result without proof; see \cite[Proposition~3.6]{mf} . 
\begin{Proposition}[\bf chain rule for semi-differentiable functions]\label{semicalc}
	Let $f: = g \circ F $ where $F:\R^n \to \R^m $ is semi-differentiable at $\ox $  and $g: \R^m \to \R^k$ is semi-differentiable at $F(\ox)$. Then,  \eqref{fcalc1} holds and  $f$ is semi-differentiable at $\ox$. 
\end{Proposition}
The chain rule \eqref{fcalc1} will lead us to the calculation of tangent vectors of constraint sets, which is a key to the first-order necessary optimality conditions for a constrained optimization problem. Next, we show an important feature of metric subregularity constraint qualification that is the metric subregularity constraint qualification imposed on a set is preserved by the tangent cone to the set.   
\begin{Corollary}[\bf preserveness of metric subregularity by tangent cones]\label{calctan} Let $\ox \in \O:= \{x \big|\; F(x) \in X \}$ where $F: \R^n \to \R^m$ is semi-differentiable at $\ox$ .If the metric subregularity constraint qualification \eqref{msqccons} holds at $\ox$, then
		\begin{equation}\label{calctan1}
			T_{\O}(\ox) = \{ w \in \R^n \big| \; \d F(\ox)(w) \in T_{X} (F(\ox))   \}
		\end{equation} 
		Moreover, the metric subregularity constraint qualification holds (globally) at $w=0$ for the constraint set $T_{\O}(\ox)$ in \eqref{calctan1} with the same constant $\kappa >0$, i.e., for all $w \in \R^n$ we have
		$$        \dist \big( w \; ; \; T_{\O}(\ox) \big) \leq \kappa \; \dist \big( \d F(\ox)(w) \; ; \; T_{X} (F(\ox))  \big)     $$
\end{Corollary}
\begin{proof}
	\eqref{calctan1} was proved in \cite[Corollary~3.11]{mf},  which is obtained by applying the subderivative chain rule \ref{fcalc1} on the composite function $\delta_{X} \circ F$ at $\ox$. Turning to the second part, fix $w \in \R^n$.  Since $F$ is semi-differentiable at $\ox$, it is directinally differentiable at $\ox$  in direction $w$.  Hence, $F(\ox + t w) = F(\ox) + t \d F(\ox)(w) + o(t)$. Moreover, for sufficiently small $t >0$,  \eqref{msqccons} yields that
	\begin{eqnarray*}
		\disp  t \; \dist ( w \; ; \;  \frac{\O - \ox}{t} ) &\leq & \kappa \; \dist (F(\ox+tw) \;  ; \; X  ) \\\nonumber
		&=& \kappa \; \dist (F(\ox)+ t \d F(\ox) (w) + o(t) \;  ; \; X  ) \\\nonumber
		&\leq & t\kappa \;  \dist  ( \d F(\ox) (w) \;  ; \;  \frac{ X - F(\ox)}{t}  ) + o(t)
		\disp
	\end{eqnarray*}
	Now by dividing both sides by $t$, we arrive at
	\begin{eqnarray*}
		\disp \liminf_{t \dn 0}  \dist ( w \; ; \;  \frac{\O - \ox}{t} ) &\leq & \kappa \; \liminf_{t \dn 0} \dist  \big( \d F(\ox) (w) \;  ; \;  \frac{ X - F(\ox)}{t}  \big) \\\nonumber
		\disp
	\end{eqnarray*}
	hence, , by \cite[Exercise~4.2]{rw}
	\begin{eqnarray*}
		\disp  \dist ( w \; ; \; \Limsup_{t \dn 0}  \frac{\O - \ox}{t} ) &\leq & \kappa \;  \dist  \big( \d F(\ox) (w) \;  ; \; \Limsup_{t \dn 0} \frac{ X - F(\ox)}{t}  \big) 
		\disp 
	\end{eqnarray*}
where the "$\Limsup$" is understood in the sense of set convergence which yields, by \cite[Proposition]{rw}. According to the equivalent definition of tangent cones by  "$\Limsup$" the proof is completed.   
\end{proof}

\begin{Theorem}[\bf semi-differentiability of the distance function]\label{semidist} Let $X$ be a nonempty closed set. Set $p(x):= \dist (x ; X)$. Then, the subderivative of $p$ is computed by
\begin{equation}\label{subdist}
\d p(x)(w) =\left\{\begin{matrix}
\dist (w \;  ;  \;  T_{X} (x)) & x \in X   \\\\  
   \min \{ \frac{\la x - y \;  , \; w\ra}{p(x)} \big|  \;y \in \proj_{X}(x) \}&   x \notin X 
\end{matrix}\right.
\end{equation}
Furthermore, the following statements hold
\begin{itemize}
\item[(1).] If $x \notin X$ then $p$ is semi-differentiable at $x.$
\item[(2).] if $x \notin X$ then $p$ is differentiable at $x$  if and only if  $\proj_{X} (x)$ is singelton, in this case $$\nabla  p(x) = \frac{1}{\dist (x ;X)} \big( x - \proj_{X} (x) \big).$$
\item[(3).] $X$ is geometrically drivable at $x \in X$ if and only if $p$ is semi-differentiable at $x \in X.$
\item[(4).]  $p$ is semi-differentiable if and only if $X$ is geometrically drivable. 
\item[(5).] If $q = \frac{1}{2} p^2$ then $q$ is semi-differentiable with the semi-derivative
\begin{equation}\label{hsub}
\d q(x)(w)=  \min \{ \la x - y \;  , \; w \ra  \big|  \; y \in \proj_{X}(x) \},
\end{equation}
moreover, for all $x, y \in \R^n$, the following upper estimate holds for $q$
\begin{equation}\label{hdes}
	q(y) \leq q(x) + \d q(x) (y-x) + \frac{1}{2} \| y -x \|^2 .
\end{equation}
\item[(6).] $q=\frac{1}{2} p^2$ is differentiable at $x$ if and only if $\proj_{X}(x)$ is singelton, in this case $$\nabla  q(x) =  x - \proj_{X} (x) .$$

\end{itemize}
\end{Theorem}
\begin{proof}
The expression for the subderivative of $p$ in \eqref{subdist} was obtained in \cite[Example~8.53]{rw}.  Fix $ x, w \in \R^n$. Since $p$ is (globally) Lipschitz continuous, proving the semi-differentiability of $p$ at $x$ reduces to the following inequality
$$  \limsup_{t \dn 0} \frac{p(x + t w) - p(x)}{t}  \leq \d p(x)(w).  $$
To prove (1), we suppose $\ox \notin X$. According  \eqref{subdist}, we can find $y \in \proj_{X} (x)$ such that $\d p(x)(w) =  \frac{\la x - y \;  , \; w\ra}{p(x)} $. Define $\eta (t):= \|x + tw - y \|$. The function $\eta: \R \to \R$ is differentiable at $t = 0$, because $x \neq y$. Since $p(x + tw) \leq \| x + tw - y \|$ for all $t \geq 0$, we have 
$$ \frac{p(x + t w) - p(x)}{t} \leq \frac{\eta(t) - \eta(0)}{t} \quad \quad t>0, $$ 
hence
$$ \limsup_{t \dn 0} \frac{p(x + t w) - p(x)}{t} \leq \eta' (0) = \frac{\la  x -  y \;  , \; w\ra}{\| x - y \|} = \d p(x) (w) $$
which proves the semi-differentiability of $p$ at $x$.  To prove (2), we already know $p$ is semi-differentiable by (1), thus the differentiability of $p$ is equivalent to linearity of $\d p(x)(.)$. Notice that $$\d p(x)(w) = \min \{ \frac{\la x - y \;  , \; w\ra}{p(x)} \big|  \;y \in \proj_{X}(x) \} $$ is linear in $w$ if and only if $\proj_{X}(x)$ is singleton.  To prove (3), first we assume  $X$ is geometrically derivable at $x$, thus we have $p(x) =0$ and $$ T_{X}(x) =\Lim_{t \dn 0}\frac{X -  x}{t}.$$ By the connection of limits and distance functions in \cite[Corollary~4.7]{rw},  we can write 
\begin{eqnarray*}
\disp
\d p(x)(w) &=& \dist (w \;  ; \;  T_{X} ( x)) =  \dist \big( w \; ; \; \Lim_{t \dn 0}\frac{X - x}{t} \big) \\\nonumber
&=& \lim_{t \dn 0} \dist \big( w \; ; \; \frac{X -  x}{t} \big) = \lim_{t \dn 0} \frac{p(x + t w) - p(x)}{t}.
\disp
\end{eqnarray*}
which proves the semi-differentiability of $p$ at $ x$. Now we assume $p$ is semi-differentiable at $x \in X.$ We fix $w \in  T_{X}(x)$, thus by \eqref{subdist}, $\d p(x)(w) = \dist (w \;  ; \;  T_{X} (x))  = 0 .$  Since $p$ is semi-differentiable at $x$ we have 
$$    0 = \d p(x)(w) = \lim_{t \dn 0} \frac{p(x + t w) - p(x)}{t}  =  \lim_{t \dn 0} \frac{   \|x + tw - y(t) \| }{t} = 0. $$
where $y(t) \in \proj_{X}(x+tw)$ for each $t >0$. The above reads as $y(t) = x + tw + o(t) $,  hence, $ x + tw + o(t) \in X$ which means $w$ is derivable tangent vector for $X$ at $x$. (4) follows from (1) and (3) by considering two cases $x \in X$ and $x \notin X.$  Turning to (5), to obtain the expression for $\d q(x)$ in \eqref{hsub},  first, we assume $x \in X$. In the latter case, we show that $q$ is differentiable with $\nabla q(x) = 0 $. Indeed, in the latter case,  $\proj_{X}(x) = \{x\}$ forcing the right side of \eqref{hsub} to equal zero. On the other hand, for all $w \in \R^n \setminus \{0\}$, we have 
$$    0 \leq \frac{q(x+ w) - q(x)}{\|w\|}  = \frac{\dist(x+w \; ; \; X)^2}{2 \|w\|} \leq \frac{\|x+w - x\|^2}{2 \|w\|} = \frac{1}{2} \|w\|  $$
which proves $q$ is differentiable at $x$ with $\nabla q(x) = 0$, thus it verifies \eqref{hsub}. Secondly, we assume $x \notin X$, thus $p = \dist(. ; X)$ is semi-differentiable at $x$ by part (1). Hence, the chain rule \eqref{fcalc} can be applied on the composite function $q = g \circ p$ with $g(t) = \frac{1}{2}t^2$. Then, the result immediately follows from \eqref{semidist}.  To prove the estimate \eqref{hdes}, note that, $q$ can be written in the form of moreau envelope of the indicator function of $\delta_{X}$, i.e., $q(x) =  \inf \{  \delta_{X}(y) +  \frac{1}{2} \|x- y\|^2 \big| \; y \in \R^n \} .$ Therefore, \eqref{hdes} immediately follows from \cite[Proposition~4.3]{mf}.  The proof of (6) immediately follows from (5); indeed,   $q$ is semi-differentiable by (5), thus $q$ is differentiable at $x$ if and only if $\d q(x)(.)$ is a linear map in \eqref{hsub} if and only if $\proj_{X} (x)$ is singleton.  
\end{proof}
\begin{Remark}[\bf comments on Theorem \ref{subdist}]{\rm
	In the framework of Theorem \ref{subdist}, since the distance function is Lipschitz continuous, the semi-differentiability of functions $p$ and $q$ is equivalent to the B(ouligand)-differentiability in the sense \cite[Definition~4.1.1]{cp}. As observed in \cite{mf},  the semi-differentiability is the key to designing a first-order descent method. The estimate \eqref{hdes} is called descent property, which was recently introduced in \cite{mf}, is responsible for convergence-guarantee of a generalized gradient descent method for the non-Clarke regular problem. (6) is a known result and can be proved directly. However, in our case, (6) is immediately followed by the semi-differentiability of $h$. Additionally, since $h$ is locally Lipschitz continuous, $q$ is differentiable almost everywhere on $\R^n$, thus (6) proves the well-known result about almost everywhere the uniqueness of the projection points on an arbitrary closed set in $\R^n$. The functions $p= \dist(. \; X)$ and $q=\frac{1}{2} p^2$ serve as penalty terms of the constraint system $F(x) \in X$, in constrained optimization problems. In almost all applications one can think of, the right-hand side constraint $X$ is geometrically drivable, which implies by (3) the semi-differentiability of $p$ and $q$. Additionally, by (6), if $X$ is convex, which is the case in nonlinear conic programming, the square penalty term $q$ is continuously differentiable with a globally Lipschitz continuous gradient. Particularly, in the framework of the nonlinear programming \eqref{nlp}, we have $X = \R^{l}_{-}  \times \{0\}^{m-l}$ thus, by a quick calculation, the penalty terms $p$ and $q$ come out as
	\begin{equation}\label{fpen}
	p(x) =   \big( \sum_{i=1}^{l} \max ^{~ ~~~2} \{x_i , 0\}  +   \sum_{j=l+1}^{m} x_j^2   \big)^{\frac{1}{2}}.
\end{equation} 
\begin{equation}\label{hpen}
			 q(x) = \frac{1}{2} \sum_{i=1}^{l} \max ^{~ ~~~2} \{x_i , 0\}  + \frac{1}{2}  \sum_{j=l+1}^{m} x_j^2  
\end{equation}     

In the framework of the nonlinear programming, as it can be checked directly, the quadratic penalty term $q$ is differentiable with the Lipschitz gradient, but it is not twice differentiable. However, this $q$ will not provide us with an exact penalty. In contrast, $p$ provides an exact penalty with the price of non-differentiability at any point on $X = \R^{l}_{-}  \times \{0\}^{m-l}$.  The square root in \eqref{fpen} may complicate the computations in each algorithm's iterations. Hence, other types of penalty terms are considered. In particular, $\ell_1$- exact penalty, which takes the form  \eqref{epnlp} in the nonlinear programming, is considered. Since \eqref{subdist} is for Euclidean distance, it cannot be directly applied to cover $\ell_1$-exact penalty. Below, we investigate the semi-differentiability of a penalty term associated with the cartesian product of several constraints. In this case we can convert $\ell_1$-penalty in \eqref{epnlp}. 
 }
\end{Remark}
\begin{Corollary}[\bf semi-differentiability of the sparable-sum of distance functions]\label{prosemidist} Let each $X_i \subseteq \R^{m_i}$ be closed set for $i=1,2,...,k$. Each vector $x \in X := X_1 \times X_2 ... \times X_k$ can be written in the form $x= (x_1 , x_2 ,..., x_k)$  with  each $x^i \in X_i$. For each $x= (x^1 , x^2 ,..., x^k)$, define $p(x) = \sum_{i=1}^{k} \dist (x^i ; X_i)$ and $q(x) = \frac{1}{2} \sum_{i=1}^{k} \dist^2 (x^i ; X_i)$. Then the following assertions hold
\begin{itemize}
	\item[(1).]  $p$ is semi-differentiable at $x \in X$ if and only if each $X_i $ is geometrically drivable at $x^i $.  
	\item[(2).]  $p$ is semi-differentiable if and only if each $X_i $ is geometrically drivable . 
	\item[(3).] I $q$ is semi-differentiable with the semi-derivative
	\begin{equation*}
		\d q(x^1, x^2, ..., x^k)(w^1, w^2,...,w^k)= \sum_{i=1}^{k} \min \{ \la x^i - y^i \;  , \; w^i \ra  \big|  \; y^i  \in \proj_{X_i}(x^i) \},
	\end{equation*}
	moreover, for all $x, y \in \R^n$, the following upper estimate holds for $q$
	\begin{equation*}
		q(y) \leq q(x) + \d q(x) (y-x) + \frac{k}{2} \| y -x \|^2 .
	\end{equation*}
	\item[(4).] $q$ is differentiable at $x$ if and only if $\proj_{X_i} (x^i) $ is singleton for each $i=1,2,...,k$, in this case $$\nabla_i  q(x) =  x^i - \proj_{X_i} (x^i)  .$$
\end{itemize}	
\end{Corollary}
\begin{proof}
	The proofs of the above assertions are similar to the ones in the Theorem \ref{semidist} taking into account that the sum of the semi-differentiable functions is semi-differentiable, and $X := X_1 \times X_2 ... \times X_k$ is geometrically derivable at $x \in X$ if and only if each $X_i$ is geometrically derivable at $x^i \in X_i$ in the point of the question.
\end{proof}
\begin{Remark}[\bf comments on Corollary \ref{prosemidist}]{\rm Corollary  \ref{prosemidist} will let us penalized mixed consstraint sets. e.g., thec constraint set  $F_1(x) \in X_1$ and $F_2(x) \in X_2$ where $X_1 := \R_{-}^m$ and $X_2 := \S_{-}^{n}$ where $\S_{-}^{n}$ is set of negative semi-definite matrices.  Moreover, the penalty term $p(x) = \sum_{i=1}^{k} \dist (x^i ; X_i)$ can cover $\ell_1$ exact penalixation in nonlinear programming \eqref{epnlp}. Indeed, in the framework of nonlinear programming \eqref{epnlp} we could write $X := (-\infty , 0]  \times ...\times (- \infty , 0] \times \{0\} \times ...  \times \{0\}. $ Then, the penalty term $p$ in Corollary  \ref{prosemidist}  comes out as
\begin{equation*}\label{l1nlpl}
		 p(x) =  \sum_{i=1}^{l} \max  \{x_i , 0\}  +  \sum_{j=l+1}^{m} | x_j |.  
\end{equation*}		
Evidently, Theorem \ref{semidist} can be seen as a special case of Corollary \ref{prosemidist} when $k=1$, but not vice the versa. Indeed, setting $X:= X_1 \times X_2 ... \times X_k$, we have $$\dist (x ; X) = \big( \dist^2 (x^1 ; X_1) + \dist^2 (x^2 ; X_2) +...+\dist^2 (x^k ; X_k) \big)^{\frac{1}{2}} $$ 
which is different from the penalty term $p$ in Corollary \ref{prosemidist}.  The converse could be true if we had Theorem \ref{semidist}  for a general norm. Unfortunately,  \eqref{subdist} only works for Euclidean distance. The properties of the Euclidean norm plays important role in the proof of \eqref{subdist}.  Moreover, the proof of the differentiability results in (2) and (6) depends on the differentiability of the norm of the space. In the following Theorem, we establish the semi-differentiability of the distance function for the space equipped with an arbitrary Fr\'echet smooth norm. Recall that a norm $\| .\|_{*}$ is called  Fr\'echet smooth norm on $\R^n$ if $\|.\|_{*}$ is Fr\'echet differentiable on $\R^n \setminus \{ 0\}$.
}
\end{Remark}

\begin{Theorem}[\bf semi-differentiability of the distance functions defined by a Fr\'echet smooth norm]\label{semidist2} 
Let $\| .\|_{*}$ be a Fr\'echet smooth norm on $\R^n$. Denote $ \xi_{*}(x)$ the gradient of $\|.\|_{*}$ at  $x \in \R^n \setminus \{ 0\}$. Let  $X$ be a closed set in $\R^n$, and  let $\dist_{*} (x ; X)$ and $\proj_{X}^{*} (x)$ be the distance of the point $x \in \R^n$ to the set $X$ and the projection of $x$ on the set $X$ with respect to the norm $\|.\|_{*}$ respectively. Then, the subderivative of the function $p(x) : =\dist_{*} (x ; X)$ is calculated by 
\begin{equation}\label{subdist2}
	\d p(x)(w) =\left\{\begin{matrix}
		\dist_{*} (w \;  ;  \;  T_{X} (x)) & x \in X   \\\\  
		 \min \{ \la \xi (x - y) ,w \ra   \big| \;    y \in \proj_{X}^{*}(x) \}&   x \notin X 
	\end{matrix}\right.
\end{equation}
Moreover, the following assertions hold.
\begin{itemize}
	\item[(1)] If $x \notin X$ then $p$ is semi-differentiable at $x.$
	\item[(2)] if $x \notin X$ then $p$ is differentiable at $x$  if and only if  $\xi \big( x - \proj_{X}^{*} (x) \big)$ is singleton, in this case $$\nabla  p(x) = \xi \big( x - \proj_{X}^{*} (x) \big).$$
	\item[(3)] If $X$ is geometrically drivable at $x \in X$ if and only if $p$ is semi-differentiable at $x \in X.$
	\item[(4)]  $p$ is semi-differentiable if and only if $X$ is geometrically drivable. 
\end{itemize}
\end{Theorem}
\begin{proof}
First, we establish the subderivative relationship for $p$ in \eqref{subdist2} when $x \notin X$. Before continuing futher, note that since norms are equivalent in finite-dimensions, for each $x \in \R^n$ the function $y \to \|x -y\|_{*}$ is coercive and  Lipschitz continuous, hence, $\proj_{X}^{*} (x)$ is a nonempty compact set in $\R^n$ for all $x \in \R^n $, also we have 
$$ \d p(x)(w) = \liminf_{t \dn 0}  \frac{p(x +tw) -p(x)}{t}   \quad \quad \mbox{for all} ~ w \in \R^n .$$ 
Therefore, the semi-differentiability of $p$ amounts to the following inequality 
\begin{equation}\label{semidiffcriteria}
	            \limsup_{t \dn 0}  \frac{p(x +tw) -p(x)}{t}   \leq    \d p(x) (w)    \quad \quad \mbox{for all} ~ w \in \R^n .
\end{equation} 
First, we establish \eqref{subdist2} for $x \notin X$, which its proof yields assertions (1) and (2) immediately.  Let $x \notin X$. Take an arbitrary $y \in \proj_{X}^{*}(x),  w \in \R^n ,$ and $t >0$.  In this case,  we have $p(x) = \|x -y\|_{*}$ and also  $p(x+tw) \leq \|x + t w -y\|_{*} $. Then, the univariate function $\eta: \R \to \R$ defined by $\eta (t) :=  \|x + t w -y\|_{*} $ is differentiable at $t =0$ as $x \neq y$, and we have  $p(x+tw) \leq \eta(t)$ for all $t \in [0 , \infty)$. Therefore, 
$$      \limsup_{t \dn 0}  \frac{p(x +tw) -p(x)}{t}   \leq   \eta'(0) = \la \xi (x - y) ,w \ra . $$
Because  $y \in \proj_{X}^{*}(x)$ was chosen arbitrary,  we established 
\begin{equation}\label{sup0}
	\d p(x)(w) \leq   \limsup_{t \dn 0}  \frac{p(x +tw) -p(x)}{t}   \leq \min \{ \la \xi (x - y) ,w \ra   \big| \;    y \in \proj_{X}^{*}(x) \}. 
\end{equation} 
To prove the other side of the above inequality, for each $t \in [0 , \infty)$, take an arbitrary $y=y(t) \in \proj_{X}^{*}(x + t w)$.  Note that the function $y: [0 , 1] \to \R^n$ is bounded because 
$$  \|y(t)\|_{*} -\|x + t w \|_{*}  \leq \|x + t w -y(t)\|_{*} \leq  \|x + t w -y(0)\|_{*} .$$
Take any sequence $t_n \dn 0$ as $n \to \infty$ such that  
\begin{equation}\label{sup1}
	\d p(x)(w)= \lim_{n \dn \infty}  \frac{p(x +t_n w) -p(x)}{t_n} \in \R. 
\end{equation} 
By passing to a subsequence, if necessary, we can assume $y(t_n) \to y$ for some $y \in \R^n$. Observe that $y \in \proj_{X}^{*}(x)$ because $p$ is continuous and we have $p(x+t_n w) =  \|x + t_n w -y(t_n)\|_{*} $ for all $n \in \N$. Since $x \notin X$, without loss of generality, we can assume  $x + \tau_n w \neq y(t_n)$ for all $n \in \N$. Moreover, according to the choice of $y(t)$, we have  $p(x) \geq \|x -y(t_n)\|_{*}$ for all $n \in \N$, thus,

\begin{equation}\label{sup1}
	\frac{p(x +t_n w) -p(x)}{t_n} \geq   \frac{ \|x + t_n w -y(t_n)\|_{*} -  \|x -y(t_n)\|_{*}}{t_n} .
\end{equation} 
Fix $n \in \N$, and define the univariate function $\eta :[0 , t_n] \to \R$ with $ \eta (\tau) = \|x + \tau w -y(t_n)\|_{*}  $.  For all $\tau \in [0, t_n]$ one has  $$  0 < \|x + t_n w -y(t_n)\|_{*}  < \|x + \tau w -y(t_n)\|_{*} $$                      
which yields the differentiability of $\eta$ on $(0 , t_n)$. Hence, by the classical mean-value Theorem, there exists $\tau_n \in (0 , t_n)$ such that 
$$ \eta' (\tau_n) = \frac{\eta(t_n)  -  \eta(0)}{t_n }   =  \frac{ \|x + t_n w -y(t_n)\|_{*} -  \|x -y(t_n)\|_{*}}{t_n} . $$
Hence,  by combining \eqref{sup1} with the above, we get 
\begin{equation}\label{sup2}
	\eta' (\tau_n) =  \la  \xi(x + \tau_n w - y(t_n)) \; , \; w \ra  \leq  \frac{p(x +t_n w) -p(x)}{t_n}. 
\end{equation} 
Note that the sequence $ \xi(x + \tau_n w - y(t_n))$ is bounded because the function $x \to \|x\|_{*}$ is globally Lipschitz continuous. By passing to a convergent subsequence, if necessary, we can assume  $ \xi(x + \tau_n w - y(t_n)) \to z $ as $n \to \infty$. Moreover, from the outer semi-continuity of the convex subdifferential we get that $z \in \sub \| . \|_{*} (x-y)$ as $x + \tau_n w - y(t_n)  \to x - y$, when $n \to \infty.$ Furthermore,    $\|.\|_{*}$ is differentiable at $x - y$, hence,  $z = \xi(x - y) $, and we get from \eqref{sup2} that
$$    \min \{ \la \xi (x - y) ,w \ra   \big| \;    y \in \proj_{X}^{*}(x) \}  \leq  \la  \xi(x - y) , w   \ra \leq  \d p(x) (w)$$
which proves \eqref{subdist2} for $x \notin X,$ as well as \eqref{semidiffcriteria}, thus we proved assertion (1) simultaneously. To prove the assertion (2), we know the semi-differentiability of $f$ at $x \notin X$, thus, $p$ is differentiable at $x \notin X$ if and only if 
$$     \d p(x)(w)  = \min \{ \la \xi (x - y) ,w \ra   \big| \;    y \in \proj_{X}^{*}(x) \}  $$ 
is linear in $w$ if and only if $\proj_{X}^{*}(x) $ is singleton. Turning to the proof of \eqref{subdist2} when $x \in  X$. Fix $w \in \R^n$, and $u \in T_{X}(x)$. Hence, $p(x) = 0$, and we can sequence $t_n \dn 0$ such that $x +t_n u + o(t_n) \in X$ for all $n \in \N.$  We have
\begin{eqnarray}\label{sub3}
	\disp  
	\d p(x)(w)  &\leq&   \liminf_{n \to \infty}\frac{f(x +t_n w) }{t_n} \\\nonumber
	&\leq&  \liminf_{n \to \infty}\frac{   \|x + t_n w -(x+t_n u + o(t_n)) \|_{*} }{t_n} \\\nonumber
	&\leq &  \lim_{n \to \infty}\frac{   \| t_n w -  t_n u + o(t_n)) \|_{*} }{t_n} = \|  w -  u \|_{*}. 
	\disp
\end{eqnarray}
$\d p(x)(w) \leq \dist_{*} (w \; ; \; T_{X}(x)). $ To prove the other side of the latter inequality, take any sequence $t_n \dn 0$ with 
\begin{equation}\label{sup3}
	\d p(x)(w)  =   \lim_{n \to \infty}\frac{p(x +t_n w) }{t_n}  = \lim_{n \to \infty}\frac{   \|x + t_n w - y_n \|_{*} }{t_n} \in \R,
\end{equation}
where $y_n$ is any vector in  $\proj_{X}^{*}(x+ t_n w) $. Without loss of generality, we can assume $\{y_n\}$ is a convergent sequence, say to $y$, as we earlier showed that such sequences are bounded. Define the sequence $u_n :=\frac{y_n - x}{t_n} $, and observe from \eqref{sup3} that the sequence $\{u_n\}$ is bounded, by passing to a convergent subsequence, we may assume $u_n \to u $ for some $u \in \R^n$. It is clear ftom definition of the Tangent cone that $u \in T_{X}(x)$, and we have $\d p(x) (w)= \|  w -  u \|_{*}$ by \eqref{sup3}. Therefore, we proved  $\d p(x) (w) \geq \dist_{*} (w \; ; \; T_{X}(x)) $ which finishes the proof of \eqref{subdist2}. To prove (3), first we assume $X$ is geometrically derivable at $x \in X$, thus we have $p(x)=0$. Take an arbitrary $w \in \R^n$ and a $u \in T_{X}(x)$ with $$  \|  w -  u \|_{*} =    \dist_{*} (w \; ; \; T_{X}(x)) =  \d p(x) (w) .$$
Because $X$ is geometrically derivable at $x$, we can write $x + tu + o(t) \in X$, thus the following holds

\begin{eqnarray*}
	\disp  
	  \limsup_{t \dn 0}  \frac{p(x +tw) -p(x)}{t}  \leq \limsup_{t \dn 0} \frac{   \|x + tw -(x+tu + o(t)) \|_{*} }{t}  =  \|  w -  u \|_{*} = \d p(x) (w),  \\\nonumber
	\disp
\end{eqnarray*}
which proves the semi-differentiability by of $p$ at $x$ by \eqref{semidiffcriteria}.  The proof of the other side of assertion (3) is basically identical to the proof of assertion (3) in Theorem \ref{semidist} thus we omit its proof. (4) follows from (1) and (3) by considering two cases $x \in X$ and $x \notin X.$ 
\end{proof}
\begin{Remark}[\bf penalty terms derived from non-Euclidean norms]\label{rne}  {\rm  Theorem \ref{semidist2} generalizes the semi-differentiability of the Euclidean distance function to any distance function derived from Fr\'echet norms, particularly, all $\ell_{\alpha}$ norms with $1 < \alpha <\infty. $  In the case of non-linear programming, i.e., $X = \R^{l}_{-}  \times \{0\}^{m-l}$, the penalty term $p(x)$ comes out as
		\begin{equation*}\label{rne1}
			p(x) =   \big( \sum_{i=1}^{l} \max ^{~ ~~~\alpha} \{x_i , 0\}  +   \sum_{j=l+1}^{m}  |x_j |^{\alpha}   \big)^{\frac{1}{\alpha}}.
		\end{equation*}
		In  the latter setting, $p^{\alpha}(x)$ is twice continuously differentiable for $\alpha \in  (2 , \infty)$, in particular, $\nabla ^2 f^{\alpha}(x)$ is (globally) Lipschitz continuous for $\alpha =3$. Similar to Corollary \ref{prosemidist}, we can generalize Theorem \ref{semidist2} to the product sets $X:= X_1 \times ... \times X_k $ while the norms used in distance functions can be different. In particular, the penalty term $q(x) := \dist^{\alpha}_{\alpha} (x ; X_1) + \dist^{\beta}_{\beta} (x ; X_2)$ where the index $\alpha$ stands for $\ell_{\alpha}$ norm used in defining $\dist_{\alpha}$. In the non-linear programming setting, the penalty term 
		\begin{equation*}\label{rne2}
			q(x) :=    \sum_{i=1}^{l} \max ^{~ ~~~3} \{x_i , 0\}  +   \sum_{j=l+1}^{m}  x_j^2 
		\end{equation*}
		is twice continuously differentiable with the globally Lipschitz continuous hessian. 
}
\end{Remark}

\section{Exact penalization}\sce \label{sec03}
Any optimization algorithm designed for solving unconstrained problems can eventually solve constrained optimization problems by penalization techniques. In numerical methods, exact penalization plays a vital role to incorporate the constraints into objective functions at the price of non-smoothness. The main advantage of exact penalization is that the penalty constant does not need to tend to infinity. In this subsection, we consider the following constrained optimization problem 
\begin{equation}\label{conso}
\mbox{minimize} \;\;  \varphi(x) \quad \mbox{subject to}\;\; F(x)  \in X,
\end{equation}
where $X$ is closed subset of $\R^m$ and $F: \R^n \to \R^m$ is semi-differentiable. The exact penalized problem associated with problem \eqref{conso} is as follows
\begin{equation}\label{pconso}
\mbox{minimize} \;\;  \varphi(x) + \rho \; \dist (F(x) \; ;  \; X) \quad \mbox{subject to}\;\; x  \in \R^n . 
\end{equation}
\begin{Definition}[\bf stationary solutions]\label{sta}
The point $\ox \in \R^n$ is called d(irectional)-stationary solution for the constrained optimization problem \eqref{conso}, if $F(\ox ) \in X$ and for all $w \in \R^n$ with $\d F(\ox)(w) \in T_{X}(F(\ox))$ one has $\d \varphi (\ox)(w) \geq 0 .$
\end{Definition}
Since we do not require differentiability and convexity assumptions for functions $\varphi$, $F$, and the set $X$, problem \eqref{conso} covers a wide range of constrained optimization problems; e.g. training a deep neural network function, optimization problems with complementarity constraints, and bilevel programming; see the section 2 in \cite{mf}. It is well-known that under a constraint qualification, any local minimizer of the constrained problem \eqref{conso} is a local minimizer of the unconstrained problem \eqref{pconso} for some sufficiently large $\rho > 0$. The latter is useful to obtain a necessary optimality condition for the constrained optimization problem \eqref{conso}; see \cite{jzz, mm21, i}. However, in practice, we solve the penalized unconstrained problem \eqref{pconso}. Therefore, it is important to investigate the link between the solutions of the unconstrained problem \eqref{pconso} and the constrained problem \eqref{conso}. Thanks to Theorem \ref{semidist}, all results in this section can be adapted to any norm used to define the distance function in \eqref{pconso}. The next two following theorems concern the latter link between solutions. As always denote $\O$, the set of all feasible points of the constrained optimization problem \eqref{conso}. 
\begin{Theorem}[\bf exact penalty]\label{ep}
Let $ \ox \in \R^n$ and $\varepsilon > 0$ be such that $\varphi$ is Lipschitz continuous on $\B_{\varepsilon}(\ox)$ with a constant $\ell \geq 0$. Furthermore, assume the metric subregularity estimate \eqref{msqccons} holds on  $\B_{\varepsilon}(\ox)$ with a constant $\kappa > 0$. Then, the following assertions hold.
\begin{itemize}
\item[(1).] if $\ox \in \O $ is a minimizer of the constrained optimization problem \eqref{conso} over $\B_{\varepsilon} (\ox)$, then for all $\rho \geq \kappa \ell$,  $\ox$ is a minimizer of the unconstrained problem \eqref{pconso} over $\B_{\frac{\varepsilon}{2}}(\ox).$ \\
\item[(2).] let $\rho > \kappa \ell$, and let $\ox $ be a minimizer of the unconstrained problem \eqref{pconso} over $\B_{\varepsilon}(\ox)$. If $\proj_{\O} (\ox) \cap \B_{\varepsilon}(\ox) \neq \emptyset$, then $\ox \in \O$  is a minimizer of the constrained optimization problem \eqref{conso} over $\B_{\varepsilon}(\ox)$.  
\end{itemize}
\end{Theorem}
\begin{proof}
Assertion (1) is a well-know result about exact penalty; see, e.g.,  \cite[page~336]{i}. To prove (2), first, we suppose that $\ox \in \O$ is a minimizer of the unconstrained problem \eqref{pconso} over $\B_{\varepsilon}(\ox)$. Therefore, for all $x \in \B_{\varepsilon}(\ox)$ we have
\begin{equation}\label{ep1}
\varphi (x) + \rho \; \dist (F(x) ;  X) \geq \varphi (\ox) + \rho \; \dist (F(\ox) ; X). 
\end{equation}
If $x \in \O \cap \B_{\epsilon}(\ox)$, then $\dist (F(x) ; X) = \dist ( F(\ox) ; X) =0 $, thus \eqref{ep1} implies $\varphi (x) \geq \varphi (\ox)$ for all $x \in \O \cap \B_{\varepsilon}(\ox)$. To complete the proof of (2), it remains to show that $\ox \in \O.$ If the latter is not true, then $\dist ( \ox , \O) > 0$. Pick any  $x \in \proj_{\O} (\ox) \cap \B_{\varepsilon}(\ox) $, then we have
$$  0 < \dist ( \ox , \O) = \|x - \ox\| \leq \epsilon. $$
Remember that $\varphi$ is Lipschitz continuous on $B_{\varepsilon} (\ox)$,  and $x \in B_{\varepsilon} (\ox)$. Therefore, from inequity \eqref{ep1} we get that 
$$    \rho \; \dist (F(\ox) ; X) \leq \varphi(x) - \varphi(\ox) \leq \ell \; \| x - \ox  \| = \ell  \; \dist ( \ox ; \O).  $$
We also know that $ \dist ( \ox ; \O) \leq \kappa \; \dist (F(\ox) ; X)  $ from metric subregularity estimation. Combining the latter two inequalities we get that $\rho \leq \kappa \ell$ which is a contradiction.    
\end{proof}
\begin{Remark}{\rm
In (2), when $\varepsilon$ is small, the condition $\proj_{\O} (\ox) \cap \B_{\varepsilon}(\ox) \neq \emptyset $ requires $\ox$ to be not too far from the feasible set $\O$. On the other hand, if $\ox$ is a global minimizer of the unconstrained problem \eqref{pconso}, then by choosing $\varepsilon > 0$ sufficiently large the condition $\proj_{\O} (\ox) \cap \B_{\varepsilon}(\ox) \neq \emptyset $ is satisfied. However, finding a global minimizer of the unconstrained problem  \eqref{pconso} can still be a challenging task unless the unconstrained problem is convex. Hence, it is important to investigate the link between problems \eqref{conso} and \eqref{pconso} by their local minimizers and d-stationary points.}
\end{Remark}
\begin{Theorem}[\bf stationary points of the penalized problem]\label{spp} Define $\dist (. ,X)$ by the Euclidean norm. Let the constants $\ell , \varepsilon, \kappa > 0$ be the same as the ones in Theorem \eqref{2.5}, and let $F$ be semi-differentiable at $\ox$. Moreover, in assertions (i) and (ii), we assume there exists a $\kappa' \geq 0$  such that 
\begin{equation}\label{spp1}
\inf_{\substack{
   \|w\| \leq 1  \\
  y \in \proj_{X} (F(\ox))
  }} {\la \d F(\ox)(w) ~ , ~ F(\ox) - y \ra} \leq - \frac{1}{\kappa' } \; \dist (F(\ox)  ;  X)
\end{equation}
Then, the following assertions hold
\begin{itemize}
\item[(1).] if $\ox $ is a local minimizer of the unconstrained problem \eqref{pconso} for $\rho > \kappa' \ell$, then $\ox \in \O$ and it is a local minimizer of the original constrained problem \eqref{conso}.\\
\item[(2).] if $\ox $ is a $d$-stationary point of the unconstrained problem \eqref{pconso} for $\rho > \kappa' \ell$, then $\ox \in \O$ and it is a $d$-stationary point of the original constrained problem \eqref{conso}.
\item[(3).] if $\ox \in \O $ is a $d$-stationary point of the constrained problem \eqref{conso}, then $\ox $ is a $d$-stationary point of the unconstrained problem \eqref{pconso} when $\rho \geq \kappa \ell$ provided that the metric subregularity constrained qualification \eqref{msqccons} holds at $\ox$ with a constant $\kappa$.    
\end{itemize}
\end{Theorem}
\begin{proof}
To prove (1), first assume that $\ox \in \O$ is a local minimizer of \eqref{pconso}. Hence, without loss of generality, inequality \eqref{ep1} holds for all $x \in \B_{\varepsilon}(\ox)$. Following the same lines of the proof in Theorem \ref{ep}(2), we obtain that $\ox$ is a local minimizer of the constrained problem \eqref{conso}. To complete the proof, it remains to show that $\ox \in \O.$ Suppose that $\ox \notin \O$ is a local minimizer of \eqref{pconso}. Therefore, $\ox$ is a $d$-stationary point of problem\eqref{pconso}. The chain rule \eqref{fcalc1} together with \eqref{subdist} help us describe $d$-stationarity of $\ox$ in terms of the initial data as follows: for all $w \in \R^n$ 
\begin{equation*}
\d \varphi (\ox)(w) + \rho \; \min_{y \in \proj_{X}(F(\ox))} \la \d F(\ox) (w) ~ , ~ \frac{F(\ox) - y}{\dist (F(\ox) ; X)} \ra  \geq 0
\end{equation*}   
Pick $w \in \R^n$ with $\|w\| \leq 1$ and $y \in \proj_{X} (F(\ox))$. By plugging these $w$ and $y$ into the above inequality we get
\begin{equation*}
\la \d F(\ox) (w) ~ , ~ F(\ox) - y \ra  \geq - \frac{\dist (F(\ox) ; X)}{\rho} \; \d \varphi (\ox)(w)
\end{equation*}  
 $\ell > 0$ is a Lipschitz constant for $\varphi$, thus $ - \ell \leq \d  \varphi (\ox)(w)$. Combining the above inequality with \eqref{spp1} we arrive at
 \begin{equation*}
  - \frac{1}{\kappa' } \; \dist (F(\ox) ; X) \geq \la \d F(\ox) (w) ~ , ~ F(\ox) - y \ra  \geq - \frac{\dist (F(\ox) ; X)}{\rho} \; \d \varphi (\ox)(w) \geq - \ell \frac{\dist (F(\ox) ; X)}{\rho}
\end{equation*}
Since $\ox \notin \O$, we have $\dist (F(\ox) ; X) > 0$, so by canceling $\dist (F(\ox) ; X) > 0$ from the above inequalities we get $\rho \leq \kappa'  \ell$, which is a contradiction.  Now turning to (2), first we suppose that $\ox \in \O$  is a $d-$stationary point of \eqref{pconso} for $\rho > \kappa' \ell$. Again by using \eqref{fcalc1} and \eqref{subdist}, for all $w \in \R^n$, we have
\begin{equation*}
\d \varphi (\ox)(w) + \rho \;   \dist \big( \d F(\ox)(w) ~;~ T_{X} (F(\ox)) \big)  \geq 0
\end{equation*} 
In particular, for all $w$ with $\d F(\ox)(w) \in T_{X} (F(\ox))$, we have $\d \varphi (\ox)(w)  \geq 0 $,  which means $\ox \in \O$ is a $d-$stationary point of the constrained problem \eqref{conso}. To finish the proof of (2), it remains to show that $\ox \in \O$. If $\ox \notin \O$, proceeding the same lines of the proof in assertion (1) will lead us to a contradiction. 
To prove assertion (3), we assume $\ox \in \O$ is a $d$-stationary point of \eqref{conso} which implies $w=0$ be the global minimizer of the following problem 
\begin{equation}\label{spp2}
\mbox{minimize} \;\;  \d \varphi(\ox)(w) \quad \mbox{subject to}\;\; \d F(\ox)(w)  \in T_{X} (F(\ox)). 
\end{equation}
Since metric subregularity constraint qualification\eqref{msqccons} holds at $\ox$, by Corollary \ref{calctan}, the metric subregulariy constraint qualification holds at $w=0$ for problem \eqref{spp2} with the same constant,  i.e, for all $w \in \R^n$ we have
\begin{equation*}\label{mspp}
\dist \big( w\; ; \; T_{\O}(x)  \big) \leq \kappa \; \dist \big(\d F(\ox)(w) \; ; \; T_{X} (F(\ox))\big)
\end{equation*}  
now by applying Theorem \ref{ep}(1) on the constrained problem \eqref{spp2}, we get $w=0$ as a local minimizer of the following unconstrained problem
\begin{equation*}
\mbox{minimize} \;\;  \d \varphi(\ox)(w) + \rho \; \dist \big(\d F(\ox)(w) \; ; \; T_{X} (F(\ox))\big) \quad \mbox{over all}\;\; w \in \R^n . 
\end{equation*}
Since objective function of the above problem is positively homogeneous, $w=0$ is in fact a global minimizer of the above problem. Hence, for all $w \in \R^n$ we have 
$$\d \varphi(\ox)(w) + \rho \; \dist \big(\d F(\ox)(w) \; ; \; T_{X} (F(\ox))\big) \geq 0$$
which means $\ox \in \O$ is a $d$-stationary point of the unconstrained problem \eqref{pconso}.
\end{proof}
Condition \eqref{spp1} is similar to a constraint qualification because it is a condition only involving the constraint set with the exception that it does not require $\ox$ to be a feasible point of \eqref{conso}. In the next result, we mention four sufficient conditions for \eqref{spp1}.
\begin{Proposition}[\bf sufficient  conditions for \eqref{spp1} ]\label{sufcon}
In the setting of the constraint optimization \eqref{conso}, assume $X$ is a closed set, and $F:\R^n \to \R^m$ is semidifferentiable at $\ox$. If one of the following conditions holds, then \eqref{spp1} holds with the specified $\kappa'$. 
\begin{itemize}
\item[(1).] if $F(\ox) \in X$, then \eqref{spp1} holds for all $\kappa' > 0 .$

\item[(2).] if $w=(u,v) \in \R^n$ and $\d F(\ox)(w)= C(u) + T(v) $, where $C(0) = 0$ and $T$ is a surjective linear map, then \eqref{spp1} holds with $\kappa' = \|T\|$.

\item[(3).] if $F$ is differentiable at $\ox$, and $w \to \nabla F(\ox)(w)$ is  surjective (LICQ), then \eqref{spp1} holds with  $\kappa' = \|\nabla G (\ox)\|$.

\item[(4).] if $X$ is convex, and there exists $\ow \in \B$ such that
\begin{equation}\label{sufcon1}
 S(\ox , \ow) := \la \d F(\ox)(\ow) ~ , ~ F(\ox) - \proj_{X} (F(\ox)) \ra < 0,
\end{equation}
then \eqref{spp1} holds with $\kappa' = - \frac{ \dist (F(\ox) ; X)}{S(\ox , \ow)}.$ 
\end{itemize}
\end{Proposition}
\begin{proof} If $F(\ox) \in X$, then we get $\proj_{X}(F(\ox)) =\{F(\ox)\}$ and $\dist (F(\ox) ; X) =0$. Hence, \eqref{spp1} holds trivially for all $\kappa' >0 .$ To prove assertions (2)-(3), we claim that there exists $\kappa' > 0$ such that for all $y \in \R^m$ we have 
\begin{equation*}
\inf_{\substack{\|w\| \leq 1}} {\la \d F(\ox)(w) ~ , ~ y \ra} \leq - \frac{1}{\kappa'} \; \|y\|.
\end{equation*}
Clearly the above inequality implies \eqref{spp1}. To prove the claim, first note that in the setting of assertion (2) we have $\|w\| = \sqrt{\|u\|^2 + \|v\|^2} $, which implies that the set of all $w=(0,v)$ with $ \|v\| \leq 1$ is a subset of the set of all $w$ with $\|w\|\leq 1$. Moreover, $T$ is a surjective linear map, thus we have $\ker (T^*) = \{0\}$ which implies $\| T^*(y)  \| \geq \frac{1}{\|T\|} \; \|y\|$. Having the latter in mind with the assumption that $C(0) = 0$, the following inequalities hold
\begin{eqnarray*}
\disp
\inf_{\substack{\|w\| \leq 1}} {\la \d F(\ox)(w) ~ , ~ y \ra} &\leq & \inf_{\substack{\|v\| \leq 1 }} {\la C(0) + T(v) ~ , ~ y \ra}\\\nonumber
&=&\inf_{\substack{\|v\| \leq 1 }} {\la T(v) ~ , ~ y \ra}\\\nonumber
&=&\inf_{\substack{\|v\| \leq 1 }} {\la v ~, ~ T^* (y) \ra} = - \|T^* (y)\| \leq \frac{-1}{\|T\|} \; \|y\|\\\nonumber
\disp
\end{eqnarray*}
by defining $\kappa' := \| T \|$, we complete the proof of the claim and assertion (2). Turning to (3), the proof is the same as the proof of (2) when we set $w=v$, $T=\nabla F(\ox)$, and ignore the partial variable $u$ and function $C$. Under the assumption in (3), also known as \textit{"LICQ"}, we obtain  $\kappa' = \| \nabla F(\ox)\|$. In (4), $X$ is closed and convex, hence, $\proj_X (F(\ox))$ is singleton. We also have
\begin{equation*}\label{sufcon2}
\inf_{\substack{
   \|w\| \leq 1  \\
  y \in \proj_{X} (F(\ox))
  }} {\la \d F(\ox)(w) ~ , ~ F(\ox) - y \ra} \leq S(\ox , \ow) =  \frac{ S(\ox , \ow)}{\dist (F(\ox) ; X)}\dist (F(\ox) ; X)
\end{equation*}
which proves (iv).  
\end{proof}
Following \cite{cp}, one application of exact penalization is to change composite problem to  additive composite problems. The latter is useful for solving deep neural network problems with non-smooth activation functions; see \cite{cp}. Recall that the unconstrained optimization problems are the optimization problems in the following form
\begin{equation}\label{comp}
\mbox{minimize} \;\;  (g \circ F)(x) \quad  \mbox{subject to} \;\;  x \in \R^n. 
\end{equation} 

\begin{Corollary}[\bf pull-out penalty]\label{pullout}
Consider the composite problem \eqref{comp}, where $F:\R^n \to \R^m$ is semi-differentiable at $\ox$ and $g$ is Lipschitz continuous around $F(\ox)$ with a constant $\ell >0$. Then, problem \eqref{comp} can be equivalently written as follows 
\begin{equation*}\label{pullout1}
\mbox{minimize} \;\;  g(y) \quad \mbox{subject to}\;\; F(x) = y ,
\end{equation*}
which its the  exact penalization has a representation of the form
\begin{equation}\label{pullout2}
\mbox{minimize} \;\;  g(y) + \rho \; \| F(x) - y \| \quad \mbox{subject to}\;\;  (x , y) \in \R^n \times \R^m.
\end{equation}

Then,  if $(\ox , \oy)$ is a d-stationary (local minimizer) point of \eqref{pullout2} with $\rho > \ell$ then $\ox$ is a d-stationary (local minimizer) point of problem \eqref{comp}.      
\end{Corollary}
\begin{proof}
Set $F(x,y):= F(x) + y$, $\varphi(x,y):= g(y)$, and $X=\{0\}^m$. Then, problem \eqref{pconso} boils down to \eqref{pullout2}. With the above setting, assertion (2) in Proposition \ref{sufcon} holds with $ T = I$. The result follows from assertions (1)-(2) in Theorem \ref{spp}.
\end{proof}
Although Theorem \ref{spp} requires the distance penalty to be defined by the Euclidean norm, one can get similar results with $\ell_1$-penalization. Indeed, $\ell_2$ and $\ell_1$ norms coincide in $\R$, hence, writing $\R^n$ as the Cartesian product of real lines equipped with Euclidean norm, gives us $\R^n$ equipped with $\ell_1$ norm. In this regard, we have the next proposition for which we do not provide a proof, as its proof is similar to \eqref{spp}. 
\begin{Proposition}[\bf $\ell_1$-exact penalty]\label{l1pen}
Suppose $ X := X_1 \times X_2 \times...\times X_m$ and $F=(f_1, f_2,...,f_m)$,  where each $X_i$ is a closed set in an Euclidean space and each $g_i$ is semi-differentiable at $\ox$. Then, a penalization of problem \eqref{conso} takes the following form
\begin{equation}\label{l1pen1}
\mbox{minimize} \;\;  \varphi(x) + \rho \; \sum_{i=1}^{m} \dist (f_i (x) ; X_i) \quad \mbox{subject to}\;\; x  \in \R^n . 
\end{equation}
Define the active index set $I(\ox):= \{i | f_i(\ox) \in X_i\}.$ Assume that there exist a $\ow \in B$ and a $y_i \in \proj_{X_i} (f_i(\ox))$ for each $i =1,2,...,m$ such that
\begin{equation}\label{l1pen2}
\left\{\begin{matrix}
     \d f_i (\ox)(\ow) \in T_{X_i} (y_i) & i \in I(\ox)   \\\\  
   S_i := \la \d f_i(\ox)(\ow) ~ , ~ f_i(\ox) - y_i ) \ra < 0  &   i  \notin I(\ox)
\end{matrix}\right.
\end{equation}
Set $$\kappa' := \left\{\begin{matrix}
     \sum_{i \notin I(\ox)} \frac{S_i}{\|f_i(\ox) - y_i \|} & if~ I^c (\ox)\neq \emptyset   \\\\  
   1  &  if~ I^c (\ox) = \emptyset
\end{matrix}\right. $$
If $\ox$ is a stationary (local minimizer) point of problem \eqref{l1pen1} with $\rho > \kappa' \ell$, then $\ox$ is a stationary (local minimizer) point of problem \eqref{conso}.  
\end{Proposition}       
\begin{Remark}[\bf comments and comparison with other results]\label{rmsc} {\rm In applications, exact penalizations can be very helpful provided that the term $\dist (x ; X)$ in \eqref{pconso} has a closed form. Luckily, in many problems, $\dist (x ; X)$ has a closed form, as $X$ usually has a nice geometric structure; for instance, $X = \R^{l}_{-} \times \{0\}^{m-l}$, $X = \S^m_{+}$, and $X = \mbox{ The Lorentz cone}.$ Furthermore, by the relationship $\dist (x ; \cup_{i=1}^s  X_i) = \min_{i} \dist (x ; X_i)$, one can compute the penalized term associated with many non-convex and Clarke-irregular constraint sets. To the best of our knowledge, there is no result similar to Theorems \ref{spp} and \ref{l1pen1} linking stationary points of the penalized problems to the original problem in such a generality. The only result we are aware of is \cite[Theorem~9.2.1]{cp} which is for the particular case $X = \R^{l}_{-} \times \{0\}^{m-l}$. While our proof is fundamentally different, which is based on calculus rules for the subderivative, in the case $X = \R^{l}_{-} \times \{0\}^{m-l}$,  one can use the active index sets to simplify the proof significantly. Even within the framework \cite[Theorem~9.2.1]{cp}, our assumptions are less restrictive; in \cite[Theorem~9.2.1]{cp}, authors further assume that $\varphi$ is directionally differentiable and $F$ is locally Lipschitz continuous. Furthermore, it is not difficult to check that the \textit{extended weak Slater CQ}, defined in  \cite[p.~529]{cp}, implies both \eqref{l1pen2} and \eqref{spp1}. Based on counterexamples in Remark \ref{subrel}, authors in \cite[Theorem~9.2.1]{cp} and a few related corollaries, need to impose the Lipschitz continuouty on a neighborhood of $\ox$, not just relative to some set. Let us highlight again the significance of the subderivative over the classical directional derivative; in Theorem \ref{spp}, we could allow $\varphi$ be an extended-real-valued function, which is Lipschitz continuous relative to its domain, provided that we take $w \in T_{\ss \dom \varphi} (\ox)$ in \eqref{spp1} and everywhere in the proof. The latter is possible due to the rich calculus of the subderivative of extended-real-valued functions, established in \cite{mf}, which is absent for the classical directional derivative. In particular, we could add $\delta_{\O} (x)$ to the objective function of problem \eqref{conso} resembling a friendly abstract constraint.   
}
\end{Remark}              
\section{Approximate stationary solutions}\sce \label{sec04}
In the previous section, we observed that the stationary solutions of a constrained optimization problem \eqref{conso} is closely related to its penalized unconstrained problem \eqref{pconso}. The latter observation motivates us to define a version of approximate stationary point for the constrained optimization problem \eqref{conso} by considering the classical approximate stationary point for the unconstrained problem \eqref{pconso}. Recall that, as widely recognized, one way to define approximate stationary point is to consider stationarity defined by an enlargement of a subdifferential. One of the widely used subgradient enlargement is , $\varepsilon$-Fr\'echet subdifferential. Accordingly, a point $\ox$ is called $\varepsilon$-stationary point for the function $f : \R^n \to \oR$ if $0 \in  \hat{\sub} _{\varepsilon} f(\ox)$  which equivalently means that $\inf_{\substack{\|w\| \leq 1}} \d f(\ox)(w) \geq  - \varepsilon .$ In latter way of defining approximate stationarity requires $\ox  \in \dom f$, meaning that in the framework of \eqref{conso} with  $f(x) := \varphi(x)  + (\delta_{X}  \circ F ) (x)  $ one need to assume $\ox$ is a feasible point, i.e.,  $F(\ox) \in X$. Most definitions of approximate stationary points in the literatures take the feasibility assumption into account; see \cite{meh1, meh2, km, ddmp, bkmw, hs}. However, from numerical view point, it is nearly impossible to find a precise solution for the system $F(x) \in X$, unless a few exceptions such as trivial solutions or when $F(x)$ is an affine mapping. To overcome the latter restriction, in our definition of approximate stationarity, we do not require $\ox$ to be a feasible solution but an approximate feasible solution; in the sense that $\dist (F(\ox)  ;  X)$ is sufficiently small. The cost of such relaxation is that, if indeed, $\dist (F(\ox)  ;  X) > 0$, then both $T_{X}(F(\ox) )$ and $N_{X}(F(\ox))$ become empty sets, thus one cannot define approximate stationarity for the constrained optimization problem \eqref{conso} directly. It seems the latter drawback was a reason that many previous articles could not take further step to define approximate stationarity for infeasible solutions for a general constraint optimization problem \eqref{conso}. To resolve the previous problem we consider approximate stationarity with the respect to the unconstrained problem \eqref{pconso} instead of \eqref{conso}. Below we formulate our novel definition of approximate stationarity.
\begin{Definition}[\bf approximate stationary solutions] \label{apps}In the setting of the constrained optimization problem \eqref{conso},  let  $X$ be geometrically derivable  and $F$ be semi-differentiable. For any  given $\varepsilon \geq 0$, we say $\ox \in \R^n$ is an $\varepsilon$-approximate stationary point for problem \eqref{conso} if there exists a $\rho \geq 0$ such that
\begin{itemize}\label{appsdef}
	\item[(1).]  $\dist (F(\ox) \;  ;  \; X)  \leq \varepsilon,$
	\item[(2).]  $\inf_{\substack{\|w\| \leq 1}}   \big(  \d \varphi (\ox)(w)  + \rho  \; \d p(F(\ox)) (\d F(\ox)w)  \big)  \geq - \varepsilon   $ 
\end{itemize}	
where $p(y) := \dist (y \; ; \;  X) $ and $\d p(y)(w)$ is the subderivative of $p$ caculated in \eqref{subdist}. 	
\end{Definition}
\begin{Remark}[\bf comments on approximate stationary solutions] {\rm If $\varphi$ is Lipschitz continuous around $\ox \in R^n$, and $\ox$ is an $\varepsilon$-approximate stationary solution for problem \eqref{conso} in the sense of definition \ref{apps} with $\varepsilon = 0$, then $\ox \in \R^n$ is indeed a stationary point for problem \eqref{conso}. To prove the latter, note that $\varepsilon = 0$ in condition (1) yields feasibility of $\ox$ while condition (2) tells us $\ox$ is a stationary solution for the unconstrained problem \eqref{pconso}. Therefore, $\ox$ is a stationary solution for problem \eqref{conso} by Theorem \ref{spp}(2) and Proposition \ref{sufcon}(1). In applications, it is helpful if we choose the parameter $\rho > 0 $ to be proportioned to $  \frac{1}{ \varepsilon}$, e.g., whenever  $\rho >> \frac{1}{ \varepsilon}$  and  $\dist (F(x) \;  ;  \; X)  \geq \varepsilon,$ then the penalty term $ \rho \; \dist (F(x) \;  ;  \; X) $ impose a significant penalty on the objective function of \eqref{pconso}. Condition (2) in Definition \ref{apps} is equivalent to $0 \in  \hat{\sub} _{\varepsilon} f_{\rho} (\ox)$ where
\begin{equation}\label{apps1}
f_{\rho} (x) = \varphi(x) + \rho \; \dist (F(x)  ;  X).
\end{equation}
To check the latter, notice that the mapping $x \tto \dist (F(x) ; X)  $ is semi-differentiable by Theorem \ref{semidist}(3) and Theorem \ref{fcalc}. Therefore,
\begin{equation}\label{apps2}
	\d f_{\rho} (x) (w) = \d \varphi (\ox)(w)  + \rho  \; \d p(F(\ox)) (\d F(\ox)w) \geq - \varepsilon
\end{equation}
or equivalently $0 \in  \hat{\sub} _{\varepsilon} f_{\rho} (\ox)$. Definition \ref{apps} for approximate stationary solutions depends on the norm we choose to define the distance function. Indeed, condition (1) will stay unchanged, up to multiple factor of $\varepsilon$, as norms are equivalent in finite-dimensions. However, condition (2) may fail for different norms. Therefore, Definition \ref{apps} to be well-defined, we consider the Euclidean norm on $\R^n$. To best of our knowledge, definitions of infeasible approximate stationary points are absent in the literature outside of nonlinear programming family. A remarkable consequence of defining an approximate (infeasible) stationary point in our way is to covering constrained optimization problems much broader than the family of nonlinear programming framework; in particular, he problems with non-Clarke regular constraints and objective functions such as bilevel programming or problems with norm-zero objectives/constraints; see \cite[Examples~2.3, 2.4, 2.5, 2.6]{mf}. Next, we specify our definition of approximate stationarity to the well-known classes of optimization problems.}
\end{Remark}
\begin{Example}[\bf approximate stationary points in non-linear conic programming] \label{nlconic} {\rm In the setting of the constrained optimization problem \eqref{conso}, assume $\varphi$ and $F$ are continuously differentiable and $X$ is a convex closed set.  If $\ox \in  \R^n$ is an $\varepsilon$-approximate stationary point for the optimization problem \eqref{conso}, then the following conditions hold 
		\begin{equation}\label{nlconic1}
			\left\{\begin{matrix}
				\dist (F(\ox) \;  ;  \; X)  \leq \varepsilon, &  \quad  \quad\\\\  
				\|   \nabla \varphi(\ox) + \nabla F(\ox)^{*}  \lambda  \| \leq \varepsilon   &  \mbox{for some} \; \lambda \in N_{X} \big(\proj_{X}(F(\ox))\big)
			\end{matrix}\right.
		\end{equation}
		Conversely, if $\ox$ satisfies \eqref{nlconic1}, then  $\varepsilon$-approximate stationary point for the optimization problem \eqref{conso}  provided that either $F(\ox) \in X$ or $ \dim N_{X} \big(\proj_{X}(F(\ox))\big) \leq 1$. }
\end{Example}       
\begin{proof}
Assume $\ox $ satisfies conditions (1) and (2) in Definition \ref{apps}  with some parameter $\rho \geq 0$. Therefore, we only need to prove the second condition in \eqref{nlconic1}. From condition (2) we have $0 \in  \hat{\sub} _{\varepsilon} f_{\rho} (\ox)$ where $f_{\rho} $ was defined in \eqref{apps1}. Since  $\varphi$ and $F$ are continuously differentiable, $f_{\rho} $  is  (Clarke) Dini-Hadamard regular by \cite[Theorem~5.2, Corollary~5.3]{mm21} . Hence, by Lemma \eqref{epsilonenlarge} we have $ 0 \in  \hat{\sub} f_{\rho} (\ox)  + \varepsilon \B$ where $\hat{\sub} f_{\rho} (\ox) $ can be calculated by the chain rule \eqref{schian}. Indeed, we have
\begin{equation}\label{nlconic2}
	0 \in \nabla \varphi (\ox) + \rho \nabla F(\ox)^* \hat{\sub} p(F(\ox))  +  \varepsilon \B.
\end{equation} 
where $p(y):= \dist (y ; X)$. Following  \cite[Example~ 8.53]{rw}, for the calculation of the subdifferential of the convex function $p(.)$, we consider two cases: 1- if $F(\ox) \notin X$. In the latter case we have $\hat{\sub} p(F(\ox)) =    \frac{1}{\dist (F(\ox) ; X)}\{  F(\ox)    -  \proj_{X}(F(\ox))\}$ thus by defining $\lambda:= \frac{\rho}{\dist (F(\ox) ; X)}  \big(F(\ox)    -  \proj_{X}(F(\ox)\big)$ clearly we have $\lambda \in N_{X} \big(\proj_{X}(F(\ox))\big)$, hence, conditions \eqref{nlconic1} holds. In the second case, we assume $F(\ox) \in X$ thus $\proj_{X}(F(\ox)) = \{ F(\ox)\}$, and we have  $\hat{\sub} p(F(\ox)) =  N_{X} (F(\ox)) \cap \B $ which obviously yields the second condition in \eqref{nlconic1}. To prove the opposite way, let $\ox \in \R^n$ satisfy \eqref{nlconic1} for some $\lambda$, thus we only need to establish \eqref{nlconic2} for some $\rho \geq 0$. If $\lambda = 0$ then evidently \eqref{nlconic2} holds with $\rho = 0$. So we assume $\lambda \neq 0$.  First let us assume $F(\ox) \in X$, hence, $\proj_{X}(F(\ox)) = \{ F(\ox)\}$, and we have  $\hat{\sub} p(F(\ox)) =  N_{X} (F(\ox)) \cap \B $, which yields $\frac{\lambda}{\| \lambda \|} \in  \hat{\sub} p(F(\ox)) $. Now \eqref{nlconic2} holds, by choosing $\rho := \| \lambda \|$. Now assume $F(\ox) \notin X$ and  $ \dim N_{X} \big(\proj_{X}(F(\ox))\big) \leq 1$. It is not hard to check that in the latter case we have $$N_{X} \big(\proj_{X}(F(\ox))\big)  =   \cone \big(   F(\ox)  - \proj_{X}(F(\ox))   \big) = \cone(\lambda).$$ 
Now \eqref{nlconic2} holds, by choosing $\rho >0 $ such that $ \lambda = \frac{\rho}{\dist (F(\ox) ; X)}  \big(F(\ox)    -  \proj_{X}(F(\ox)\big)$.
\end{proof}
\begin{Example}[\bf approximate stationary points in non-linear programming] \label{egkkt} {\rm In the setting of non-linear programming \eqref{nlp}, define $F:= (g_1,...,g_{l},h_{l+1},...,h_{m})$ and $X= \R^{l}_{-} \times \{0\}^{m-l}$. Then, the approximate stationary points in the sense \eqref{nlconic1} reads as the existence of the  multipliers $(\lambda_1,...,\lambda_{l}, \mu_{l+1},...,\mu_{m})  \in \R^{m}$ such that 
\begin{equation}\label{appskkt}
	\left\{\begin{matrix}
		\| \nabla \varphi(\ox) + \sum_{i=1}^{l} \lambda_i  \nabla g_i (\ox) + \sum_{j=l+1}^{m}   \mu_j  \nabla h_j (\ox)\| \leq \varepsilon & \\\\
		g_i(\ox) \leq \varepsilon ~i=1,...,l,\quad |  h_j (\ox) |  \leq \varepsilon ~ j=l+1,...,m&   \\\\
		\lambda_i \geq  0 ~i=1,...,l,\quad   \mu_j  \in \R ~ j=l+1,...,m&   \\\\
			\lambda_i  g_i (\ox)  = 0, \quad i \in I_{\leq} (\ox):=\{i\big|\;g_i (\ox) \leq 0\}&   \\
	\end{matrix}\right.
\end{equation}
}
\end{Example}
\begin{proof}
Let $\ox \in \R^n$ satisfy \eqref{nlconic1} for some $\varepsilon \geq 0$. It  is clear that for all $i$ and $j$ we have 
$$   \max\{g_i(x) , 0\} \leq \dist (F(\ox)  ; X) \leq \varepsilon, \quad     | h_j (\ox)| \leq \dist (F(\ox)  ; X) \leq \varepsilon,$$
which proves the approximate primal feasibility in \eqref{appskkt}.  For any $y \in X$, one can verify that
$$  N_{X}(y) = \{  (\lambda_1,...,\lambda_{l}, \mu_{l+1},...,\mu_{m})  \in \R^{m} \; \big| ~ 	\lambda_i \geq  0, ~ y_i \lambda_i = 0 ~i=1,...,l\} .$$
Additionally, due to structure of $X$, the projection over $X$ gets a simple closed form. Indeed, for each $y \in \R^m$ we have $$   \proj_{X} (y) = \big(   y_i  \; \mbox{if} \;   1 \leq i \leq l, \; y_i \leq 0  \; \textbf{\big|}\; 0  \; \mbox{otherwise} \; \big) .$$ 
Therefore, the other conditions in \eqref{appskkt} involving Lagrangian multipliers follows from the second condition in \eqref{nlconic1}. Conversely, assume $\ox \in \R^n$ satisfies \eqref{nlconic1} for some $\varepsilon \geq 0$, then $\ox$ satisfies \eqref{nlconic2} for with $m \varepsilon $. Indeed, due to description of the Normal cone to and the projection on $X$, condition \eqref{nlconic1} holds with the same $\varepsilon$. However, for approximate feasibility we can write
$$       \dist (F(\ox)  ; X) \leq   \sum_{i=1}^{l} \max\{g_i(x) , 0\}  + \sum_{j=l+1}^{m}   | h_j (\ox)|  \leq  m \varepsilon .$$
\end{proof}
\begin{Remark}[\bf comments on approximate stationary points] \label{reapp} {\rm As Examples \ref{nlconic} and \ref{egkkt} show, the approximate stationary solutions in the sense of \eqref{nlconic1} and \eqref{appskkt} are implied by \eqref{appsdef}, thus any numerical method computing an approximate stationary solution, will compute approximate stationary points in the sense \eqref{nlconic1} and \eqref{appskkt}; see the next sub-section. The approximate stationary solution for the non-linear conic programming in the sense \eqref{nlconic1} is new in the optimization literature and was not defined outside of non-linear programming. One could consider a more relaxed version of \eqref{nlconic1} by considering the $\varepsilon-$enlargement of the convex Normal cone. Indeed, replacing the normal cone with in the second condition of \eqref{nlconic} by $$  N_{X}^{\varepsilon}  \big(\proj_{X}(F(\ox))\big) := N_{X} \big(\proj_{X}(F(\ox))\big) + \varepsilon\B, $$  
one can allow approximate Lagrange multipliers in \eqref{nlconic1}, hence, as a consequence, the dual feasibility and complementarity slackness in \eqref{appkkt} can be rewritten in the approximate form, as we mentioned in the introduction section. This kind of approximate stationary solutions was previously investigated in the non-linear programming framework; see \cite{ams}. Some feasible versions of approximate stationarity solutions have already been investigated in the literature; s see \cite{ams, aghmr, meh1, meh2, bkmw}.   
}
\end{Remark}
\subsection{Computing approximate stationarity points}
In this sub-section, we aim to design a practical numerical algorithm to compute an $\varepsilon$-approxiate stationarity point for a given $\epsilon > 0$. In section \ref{sec03}, we observed that there is a significant link between stationary points of the constrained optimization problem \eqref{conso} and the unconstrained optimization problem \eqref{pconso}. The latter suggests that in order to obtain a stationary point for the constrained optimization problems \eqref{conso}, we can find a stationary point for the unconstrained problem  \eqref{pconso}. This method is unfavorable in two senses: 1- The obtained solution might not be feasible or even approximately feasible, 2- the unconstrained problem \eqref{pconso} is generally non-differentiable despite the smoothness of data of the problem \eqref{conso}.  To overcome the previous two drawbacks first we utilize higher-order penalties to reduce non-smoothness and we seek a stationary solution to the unconstrained problem in a specific sub-level set. In the following theorem, we will make our statement precise. Before continuing further, in this subsection, we make two basic assumptions on the constrained  problem \eqref{conso};
\begin{itemize}
\item[(\textbf{A1)}] In the setting of the the constrained optimization problem \eqref{conso} assume that $\varphi$ is lower semi-continuous, $F$ is semi-differentiable, there exists $x_0 \in \R^n$ such that  $F(x_0) \in X .$\\
\item[(\textbf{A2)}]There exist parameters $M \geq 0,  \rho_{0}  \geq 0,$ and $\alpha \geq 2$ such that for all $x \in \R^n$
$$      -M \leq  f_{\alpha , \rho_{0}} (x):= \varphi (x) + \rho_{0}  \; \dist^{\alpha} (F(x)  \; ; \;  X)  $$    
\end{itemize}
(A1) is a standard assumption about continuity and feasibility while a sufficient condition for (A2) is that $\varphi$ is being bounded below on $\R^n .$ We are mainly interested in (A2) holds for $\alpha = 2$. 
\begin{Theorem}[\bf finding approximate stationary points by solving a smooth equation]\label{ss} In the setting of problem \eqref{conso}, assume assumptions (A1) and (A2) hold.  Given $\varepsilon >0$, pick $\rho > 0$ sufficiently large such that $\rho \geq  \rho_{0}  +  \frac{\varphi(x_0) + M}{ \varepsilon^{\alpha}}.$  Now consider the following optimization problem
\begin{equation}\label{ss1} 
\mbox{minimize} \;\;  f_{\alpha , \rho}(x) =  \varphi(x) + \rho \; \dist^{\alpha} (F(x) \; ; \ ;X) \quad \mbox{subject to}\;\;  x \in \R^n.	 
\end{equation}
Then, the following assertions hold.
\begin{itemize}
\item[(1).] The unconstrained problem \eqref{ss1} has a $\varepsilon$-approximate stationary point in the set $S_{\alpha ,\rho}(x_0) := \{ x \big|\; f_{\alpha , \rho}(x)\leq \varphi(x_0) \}$ \\
\item[(2).] Every $\varepsilon$-stationary point for the unconstrained problem \eqref{ss1} in $S_{\alpha ,\rho}(x_0)$ is a $\varepsilon$-approximate statinary point for the constrained optimization \eqref{conso} 
\end{itemize}
\end{Theorem}
\begin{proof} Note that we implicitly exclude the trivial case $  \varphi(x_0) + M  \leq 0$, otherwise $x_0$ becomes a global minimzer of \eqref{conso} by (A2). To establish (1), without loss of generality, assume $  \varepsilon \leq \frac{1}{ \sqrt{2 \varphi(x_o) +  2 M} }$.  Consider the following auxillary optimization problem 
\begin{equation}\label{ss2}
	\mbox{minimize} \;\;  f_{\alpha , \rho}(x)  +  \frac{ \varepsilon^4}{2}   \|x - x_0\|^2 \quad \mbox{subject to}\;\;  x \in \R^n.
\end{equation}
The objective function of \eqref{ss2} is lower semi-continuous by (A1) and coercive by (A2), thus it has a global minimizer, say $\ox  \R^n$. We claim that $\ox$ is a desired vector. Indeed, $$f_{\alpha , \rho}(\ox)  +   \frac{ \varepsilon^4}{2}   \| \ox - x_0\|^2  \leq  f_{\alpha , \rho}(x_0)   = \varphi(x_0)$$ 
Not only the above proves that $f_{\alpha , \rho}(\ox) \leq \varphi(x_0) $ but also since $- M \leq f_{\alpha , \rho}(\ox) $ we get the following uper bound for $\|\ox - x_0\|$
\begin{equation}\label{ss3}
	  \| \ox - x_0\|  \leq \frac{\sqrt{2\varphi (x_0) + 2 M}}{\varepsilon^2}.
\end{equation}
Turning to prove  the statinarity of $\ox$, note that $\ox $ is a global minimizer of \eqref{ss2} so is a stationary point of \eqref{ss2}, i.e.,
\begin{equation*}\label{ss4}
0 \leq \d f_{\alpha , \rho}(\ox)(w) + \varepsilon^4  \la \ox -x_0 \; , \; w  \ra   \quad \quad \mbox{for all} ~ w \in \R^n,
\end{equation*}
hence, by applying the cauchy schwarz inequality and combining above with \eqref{ss3}, we arrive at
\begin{equation*}\label{ss5}
	0 \leq \d f_{\alpha , \rho}(\ox)(w) + \varepsilon^4  \| \ox -x_0 \|  \| w  \| \leq  \d f_{\alpha , \rho}(\ox)(w) + \varepsilon^2  \sqrt{2\varphi (x_0) + 2 M} \| w  \|.
\end{equation*} 
Acording to the choice of the $\varepsilon > 0$, we get 
\begin{equation*}\label{ss6}
	0 \leq \d f_{\alpha , \rho}(\ox)(w) + \varepsilon  \| w \|  \quad \quad \mbox{for all} ~ w \in \R^n,
\end{equation*}
meaning that $\ox$ is a $\varepsilon$-approximate stationary point for problem \eqref{ss1}. To prove (2), let us first recall the function $ y \to \dist^{\alpha} (y ;  X)$ is semi-differentiable as it can be viewed as a composition of the differetiale function $t \to t^{\frac{\alpha}{2}}$ and $\dist^{2} (y ; X)$; see Theorem \ref{semidist} (5). Therefore, thanks to the chain rule \ref{fcalc1},  the function $ x \to \dist^{\alpha} (F(x)  ;  X)$ is semi-differentibale on $\R^n$.  Now turninig to the proof of (2), assume $\ox$ is a $\varepsilon$-approximate stationary point for problem \eqref{ss1} satisfying $f_{\alpha , \rho}(\ox) \leq \varphi(x_0) $. Based on previous discussion, the subderivative of $f_{\rho} $ can be calculated by the chain rule \eqref{fcalc1}, thus for all $w \in \R^n$ with $\|w\| \leq 1$ we get
\begin{equation*}\label{ss7}
\d f_{\alpha , \rho}(\ox)(w)	= \d \varphi (\ox)(w) + \alpha \rho \; \dist^{\alpha - 1} (F(\ox) ;  X) \; \d p (F(\ox))(\d F(\ox)(w)) \geq  - \varepsilon
\end{equation*}
where $p(y) : = \dist(y ; X)$.  Thus, we only need to show that $\dist (F(\ox)  ;  X) \leq \varepsilon$. Indeed,  by assumption (A2) we have
\begin{equation*}\label{ss8}
 -M \leq f_{\alpha , \rho_{0}} (\ox) =  f_{\alpha , \rho} (\ox) + (\rho_{0} - \rho) \; \dist^{\alpha} (F(\ox)  \; ; \;  X) 
\end{equation*}
By combining the above inequality with $f_{\alpha , \rho}(\ox) \leq \varphi(x_0) $, and considering the choice of the $\rho$ we arrive at
\begin{equation*}\label{ss8}
   \dist^{\alpha} (F(x)  \; ; \;  X)  \leq   \frac{\varphi(x_0) + M}{\rho - \rho_{0}}  \leq \varepsilon^{\alpha}
 \end{equation*}
which completes the proof of the theorem. 
\end{proof}
\begin{Remark}[\bf solving constrained optimization problems by the means of Theorem \ref{ss} ] \label{ssremark} {\rm Theorem \eqref{ss} provides a practical way to compute approximate stationary points for constrained optimization problems. More precisely, to compute an $\varepsilon$-approximate stationary point for problem \eqref{conso} when an feasible solution $x_0$ is in hand, one can set up the unconstrained problem \eqref{ss1} with enough large penalty parameter $\rho$. Therefore, by running any descent method on problem \eqref{ss1}, initialized from $x_0$,  after finite step we reach at a $\varepsilon$-approximate stationary point for problem \eqref{ss1}, which is a $\varepsilon$-approximate stationary point for problem \eqref{conso}  by Theorem \eqref{ss1}(2). In the framework of non-linear conic programming, fuction $f_{\rho}(x)$ in \eqref{ss1} is continuously differentiable so the problem is reduced to solving the equation $\nabla f_{\alpha , \rho}(x) = 0$ (approximatly), which can be done by applying the gradient descent method on $f_{\rho} (x)$ initialized at $x_0$ or semi-smooth newton method on the equation $\nabla f_{\alpha , \rho}(x) = 0$; see \cite{fp}. Very recentldy in \cite{mf}, a generalized version of the gradient descent algorithm, named subderivative method, was proposed to solve unconstrained optimization problems who suffers Clarke irregularity. Next, we combine Theorem \ref{ss} with the subderivative method to propose a first-order method to solve the constrained optimization problem \eqref{conso}. 
}
\end{Remark}
\vspace{.1cm}
 \begin{table}[H]
        \renewcommand{\arraystretch}{1.25}
        \centering
        \label{alg1}
        \begin{tabular}{|
                p{0.9\textwidth} |}
            \hline
            \textbf{The Subderivative Method}\\
 
            \begin{enumerate}[label=(\alph*)]
            
                \item[0.] \textbf{(Initialization)} Pick the tolerance $\epsilon \geq 0$, the penalty parameter $\rho$, the starting point $x_0 \in \R^n$, and set $k=0 .$
           
                \item[1.] \textbf{(Termination)} If $\min_{\|w\| \leq 1}  \d f_{\alpha , \rho}(x_k)(w) \geq  - \varepsilon $ then Stop.
                \item[2.] \textbf{(Direction Search)} Pick $w_k \in \mbox{arg min}_{\|w\| \leq 1 }   \d f_{\alpha , \rho}(x_k)(w).$
                \item[3.] \textbf{(Line Search)} Choose the step size $\alpha_{k}  > 0$ through a line search method.
                \item[4.] \textbf{(Update)} Set $x_{k+1} : = x_k + \alpha_k  w_k$ and $k+1  \leftarrow k$ then go to step 1. 
            \end{enumerate}
            \\ \hline
        \end{tabular}
    \end{table}
\vspace{.1cm} 
One of the most efficient and practically used line searches is \textit{Armijo backtracking} method which can be adapted to the subderivative version; indeed, fix the parameter $\mu \in (0 , 1)$, called a reduction multiple, and assume that we are in the $k^{th} -$iteration. The Armijo backtracking line search determines the step-size $\alpha_k >0$ in the following way:
If the following inequality holds for $\alpha_k = 1$, then the step-size is chosen $\alpha_k =1$.  
\begin{equation}\label{backtrine}
	f(x_k + \alpha_k w_k) - f(x_k) <\frac{\alpha_k}{2}  \d f(x_k) (w_k).
\end{equation}
Otherwise, keep updating $\alpha_k$ by multiplying it by $\mu$ until the above inequality holds. In \cite[Lemma~4,6]{mf}emma, the author showed that the backtracking method terminates after finite numbers of updates under some mild assumptions which are the case in our framework.

\begin{Theorem}[\bf global convergence of the subderivative method with the Armjo backtracking line search ]\label{rate_TH} Let assumptions (A1) and (A2) hold with $\alpha =2$. Additionally, assume that $F$ is differentiable with Lipschitz continuous derivative and $\varphi$ satisfies the descent property \eqref{hdes}. Let $\{x_k\}_{k =0}^{\infty}$ be a sequence generated by the subderivative method with the Armijo backtracking line search. Then, after atmost $O(\frac{1}{\varepsilon^2})$ iterations, the subdervative method finds a $\varepsilon$-approximate stationary point for the constrained optimization problem \eqref{conso}. 
\end{Theorem}	 
\begin{proof} Thanks to Theroem \ref{ss}, we only need to show $ \inf_{\|w\| \leq 1} \d f_{2, \rho}(x_k) (w) \geq - \varepsilon$ for some $k= O(\frac{1}{\varepsilon^2})$.  The latter immedietly follows from \cite[Theorem~4.7]{mf} because the function $$f_{2, \rho} (x) = \varphi (x) + \big([e_{2\rho} \delta_{X}] \circ F\big)(x)  $$ satisfies the descent property \eqref{hdes}; see \cite[Proposition~4.3]{mf}. 	
\end{proof}	
\begin{Remark}[\bf comments on the convergence subderivative method]{\rm If $\varphi$ is semi-differentiable then the subderivate method applied on the function $f_{2, \rho}(.)$ is descent method \cite[Remark~4.1]{mf}, thus for the convergence analysis we do  not need to require $\nabla F$ to be globally Lipschitz continuous. Indeed, all we need is that $f_{2,\rho} (.)$ satisfies the descent property on the sublevel set $S(x_0)$, defined in Theorem \ref{ss}; see \cite[Theorem~4.8]{mf}. For instance, if $f_{2, \rho} (.)$ is coercive and $F$ is twice continuously differentiable then  $f_{\rho} (.)$ satisfies the descent property on any of its sublevel set. It is worth to mention that the assumptions made in Theorem \ref{rate_TH} are weak enough to allow us cover some classes of constrained optimization problems with non-smooth objective and constraint functions. In particular, constrained optimization problems in the framework
		\begin{equation*}\label{lr} 
			\mbox{minimize} \;\;  \varphi(x)  \quad \mbox{subject to}\;\;  (h \circ F)(x)  \leq b	 
		\end{equation*}		 
		where $\varphi$ is coercive and satisfying the descent property \eqref{hdes},  $F$ is twice continuously differentiable, and $h$ is a lower semi-continous function (possibly non-smooth and non-convex). By setting $X := \{ y \in \R^m \big| \; h(y) \leq b \}$, the above problems boils down to \eqref{conso}. In the previous example functions $\varphi$ and $h$ can be non-Clarke regular even discontinuous; e.g., $\varphi$ can be a non-smooth concave function and $h$ can be the norm-zero and rank functions; see \cite[Proposition~4.3, Example~4.11]{mf}  and \cite{cp}.  
}          
\end{Remark}    


\end{document}